\documentclass[preprint]{elsarticle}
\journal{Journal of Computational Physics}
\usepackage[left=2.7cm,right=2.7cm,top=2.5cm,bottom=2.5cm]{geometry}
\usepackage{amsthm}
\usepackage[figuresright]{rotating}
\usepackage{color}
\usepackage{pdfcomment}
\usepackage[english]{babel}
\usepackage{graphicx} 
\usepackage[cmex10]{amsmath}
\usepackage{amsfonts,amssymb,mathtools}
\usepackage{amstext}
\usepackage{upgreek}
\usepackage[tight,footnotesize]{subfigure}
\usepackage[tight,nice]{units}
\usepackage{algorithm2e}
\usepackage{import}
\usepackage{tikz}
\usepackage{pgfplots}
\usepgfplotslibrary{groupplots}
\usepackage[normalem]{ulem} \usepackage{soul} 
\numberwithin{equation}{section}
\newtheorem{thm}{Theorem}
\numberwithin{thm}{section}

\newtheorem{prop}[thm]{Proposition}

\newtheorem{rem}[thm]{Remark}
\newtheorem{prob}{Problem}
\numberwithin{prob}{section}

\graphicspath{{figures/}}

\newcommand{\bx}{\mbox{\boldmath$x$}}
\newcommand{\by}{\mbox{\boldmath$y$}}
\newcommand{\bn}{\mbox{\boldmath$n$}}
\newcommand{\bu}{\mbox{\boldmath$u$}}
\newcommand{\bv}{\mbox{\boldmath$v$}}

\newcommand{\ba}{\mbox{\boldmath$a$}}
\newcommand{\be}{\mbox{\boldmath$e$}}

\newcommand{\bh}{\mbox{\boldmath$h$}}
\newcommand{\bd}{\mbox{\boldmath$d$}}
\newcommand{\bb}{\mbox{\boldmath$b$}}
\newcommand{\bj}{\mbox{\boldmath$j$}}
\newcommand{\bM}{\mbox{\boldmath$M$}}

\newcommand{\bW}{\mbox{\boldmath$W$}}

\newcommand{\bH}{\boldsymbol{\mathcal{H}}}

\newcommand{\bJ}{\boldsymbol{\mathcal{J}}}
\newcommand{\bF}{\boldsymbol{\mathcal{F}}}
\newcommand{\bR}{\boldsymbol{\mathcal{R}}}

\newcommand{\ah}{\mbox{$\alpha$}}
\newcommand{\bah}{\mbox{\boldmath$\alpha$}}
\newcommand{\Ltwo}[2][]    {L^2#1(#2)}

\newcommand{\Hone}[2][]    {H^1#1\left(#2\right)}

\newcommand{\Hcurl}[2][]   {\boldsymbol{H}#1(\CurlSymb;#2)}

\newcommand{\GradSymb}{\mathrm{grad}}
\newcommand{\CurlSymb}{\mathrm{curl}}
\newcommand{\DivSymb}{\mathrm{div}}
\newcommand{\Curl}[2][]  {\mathrm{\CurlSymb}{#1}\,{#2}}
\newcommand{\Grad}[2][]  {\mathrm{\GradSymb}{#1}\,{#2}}
\newcommand{\Div}[2][]   {\mathrm{\DivSymb}{#1}\,{#2}}
\newcommand{\isur}[3][]{\left\langle#2,#3\right\rangle#1}
\newcommand{\ivol}[3][]{\left(#2,#3\right)#1}

\begin{document}
\begin{frontmatter}
\title{Waveform Relaxation for the Computational Homogenization of Multiscale Magnetoquasistatic Problems}
\author[temf,gsc]{I. Niyonzima\corref{cor1}}
\author[ace]{C. Geuzaine}
\author[temf,gsc]{S. Sch{\"o}ps}
\cortext[cor1]{Corresponding author.}
\cortext[cor2]{Email addresses: niyonzima@gsc.tu-darmstadt.de (I. Niyonzima), cgeuzaine@ulg.ac.be (C. Geuzaine), schoeps@gsc.tu-darmstadt.de (S. Sch{\"o}ps)}
\address[temf]{Institut f{\"u}r Theorie Elektromagnetischer Felder, Technische Universitaet Darmstadt, \\
Schlossgartenstrasse 8, D-64289 Darmstadt, Germany.}
\address[gsc]{Graduate School of Computational Engineering, Technische Universitaet Darmstadt, \\
Dolivostrasse 15, D-64293 Darmstadt, Germany.}
\address[ace]{Applied and Computational Electromagnetics (ACE), Universit\'{e} de Li\`{e}ge, \\ 
Montefiore Institute, Quartier Polytech 1, All\'{e}e de la d\'{e}couverte 10, B-4000 Li\`{e}ge,  Belgium}

\begin{abstract}
    This paper proposes the application of the waveform relaxation method to the homogenization of multiscale
magnetoquasistatic problems. In the monolithic heterogeneous multiscale method, the nonlinear macroscale problem is solved
using the Newton--Raphson scheme. The resolution of many mesoscale problems per Gau{\ss} point allows to compute the
homogenized constitutive law and its derivative by finite differences. In the proposed approach, the macroscale problem
and the mesoscale problems are weakly coupled and solved separately using the finite element method on time intervals for
several waveform relaxation iterations. The exchange of information between both problems is still carried out using the
heterogeneous multiscale method. However, the partial derivatives can now be evaluated exactly by solving only one
mesoscale problem per Gau{\ss} point.
\end{abstract}


\end{frontmatter}

\section{Introduction}
\label{intro}
The recent use of the heterogeneous multiscale method (HMM \cite{e-hmm-03}) in
electrical engineering has allowed to accurately solve magnetoquasistatic (MQS)
problems with multiscale materials, e.g. microstructured composites with
  ferromagnetic inclusions exhibiting hysteretic magnetic behavior
\cite{niyonzima-13,niyonzima-14}. The method requires the solution of one
macroscale and mesoscale problems at each Gau{\ss} point of the
macroscale problem (see Figure \ref{FE2_ppe}) in a coupled formulation based on
the Finite Element (FE) method. In \cite{niyonzima-13,niyonzima-14} the coupled
problem was monolithically time discretized by using equal step sizes at all
scales and the resulting nonlinear problem was solved by an inexact parallel 
multilevel Newton-–Raphson scheme.The finite-difference approach involves the resolution
of 4 mesoscale problems in the three-dimensional case (respectively 3 mesoscale
problems in two-dimensions) for computing the approximated Jacobian
at each Gau{\ss} point.

The use of different time steps becomes important for problems involving
different dynamics at both scales.  In the case of the soft ferrite
  material studied in \cite{bottauscio-13}, for example, it was shown that
capacitive effects occurring at the mesoscale could be accounted for by
upscaling proper homogenized quantities in the macroscopic MQS formulation.
Another relevant case involves perfectly isolated laminations and soft magnetic
composites (SMC) with eddy currents at the mesoscopic level (scales of the
sheet/metalic grain) but without the resulting macroscopic eddy currents. The
application of the HMM to problems involving such materials leads to a
formulation featuring magnetodynamic problems at the mesoscopic scale and a
magnetostatic problem at the macroscopic level. Thus, small time steps should be
used at the mesoscale to resolve the eddy currents (especially with saturated
hysteretic materials) while large time steps could be used to discretize
the rather slowly-varying exciting source current at the macroscale level.
Obviously, in such cases of different dynamics, the use of different time steps
can help to reduce the overall computational cost.
\begin{figure}
\centering
    \includegraphics[width=0.65\textwidth]{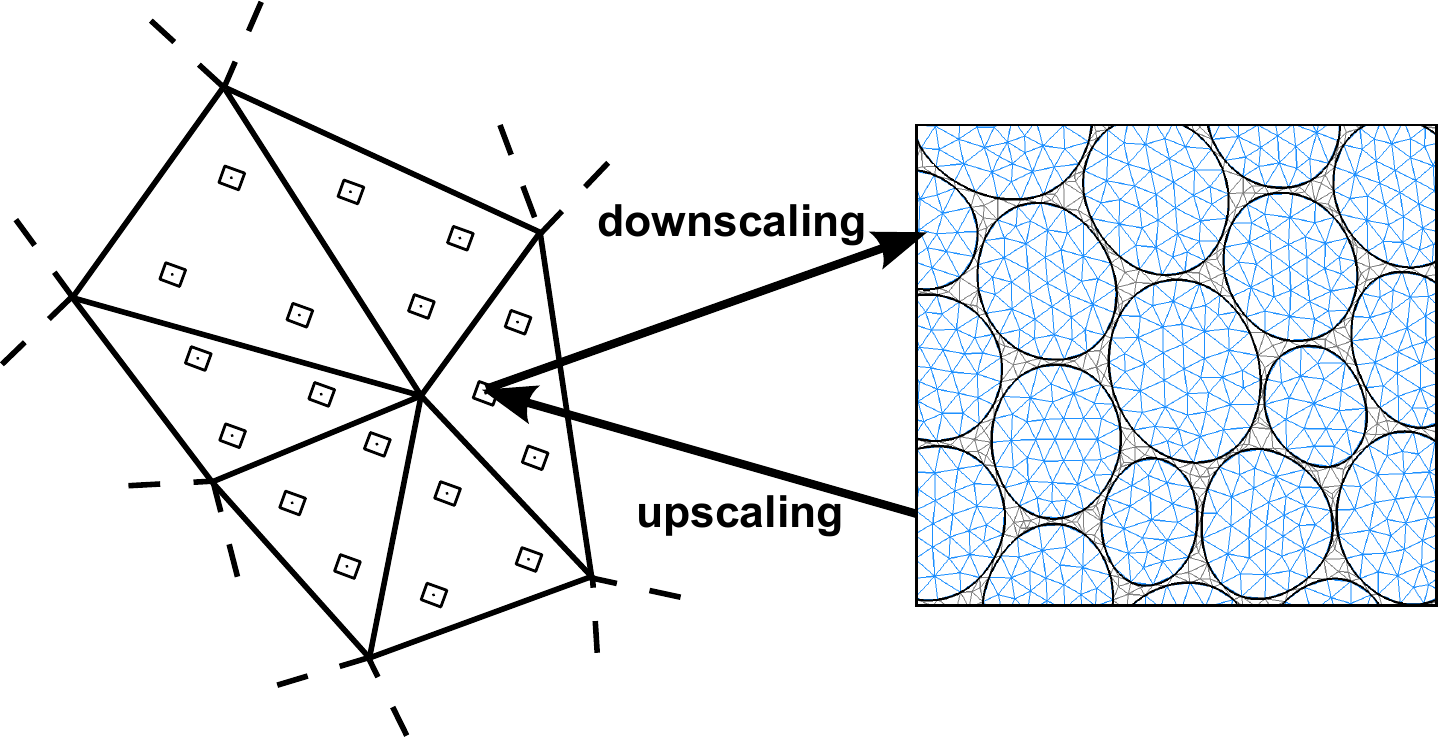}
    \caption{Scale transitions between macroscale (left) and mesoscale (right)
    problems. Downscaling (Macro to meso): obtaining proper boundary conditions and the source terms
    for the mesoscale problem from the macroscale solution. Upscaling
    (meso to Macro): effective quantities for the macroscale problem
    calculated from the mesoscale solution \cite{niyonzima-13,niyonzima-14}.}
  \label{FE2_ppe}
\end{figure}

In this paper we propose a novel approach that provides a natural setting for
the use of different time steps. The approach applies the waveform relaxation
method \cite{white-85,Schops_2010aa} to the homogenization of MQS problems: the
macroscale problem and the mesoscale problems are solved separately on time
intervals and their time-dependent solutions are exchanged in a fixed point
iteration. The decoupling of the macroscale and the mesoscale problems and the independent
resolution of these problems on time intervals has the potential to
significantly reduce both the computation and communication cost of the
multiscale scheme; in particular it allows to compute the Jacobian exactly at
each Gau{\ss} point of the macroscale domain by solving only one mesoscale
problem. As a drawback, waveform relaxation iterations are needed for the
overall problem to converge in addition to the Newton--Raphson iterations on the
meso- and macroscale.  The latter exhibits quadratic convergence, while the
fixed point iteration only leads to a linear convergence but is applied to
waveforms instead of classical vector spaces.  We present both approaches and
compare the computational and the communication costs for both the
monolithic and the waveform relaxation HMM.

The article is organized as follows: in Section \ref{sec:mqs-problem} we
introduce Maxwell's equations and the MQS problem. The weak form of the MQS
problem is then derived using the modified vector potential formulation.
Section \ref{sec:classical-hmm} deals with the multiscale formulations of the
 HMM for the MQS problem 
along the lines of the works \cite{niyonzima-13,niyonzima-14}
with an emphasis on the coupling between the macroscale and the mesoscale problems. 
These formulations are valid 
for the monolithic and the waveform relaxation (WR) HMM. In Section \ref{fe-hmm-mono}
we develop a novel theoretical framework for the monolithic HMM. Using this framework we derive 
a reduced Jacobian from the Jacobian of the full problem using the Schur complement, 
similar as it has been proposed for the Variational Multiscale Method in \cite{bottauscio-13}.
Section \ref{sec:waveform} gives a short overview of the waveform relaxation
method. The notion of weak and strong coupling are explained in the general
context of coupled systems.  The method is then used in Section
\ref{sec:waveform-homogenization} in combination with the HMM and gives
rise to the newly developed WR--HMM.  
Section \ref{sec:computational-cost} is dedicated to the estimation of the 
computational cost for both the monolithic HMM and the WR-HMM. Formulae for the 
computation of costs for the monolithic HMM and the WR-HMM are derived and analyzed 
to give a hint on a possible reduction of the computational cost of both methods.
Section \ref{sec:applications}
deals with an application case. We consider an application involving idealized 
soft magnetic materials (SMC) without global eddy currents. Convergence of
the method as a function of the waveform relaxation iterations and the
macroscale/mesoscale time stepping is numerically investigated.

\section{The magnetoquasistatic problem}
\label{sec:mqs-problem}
In an open, bounded domain
$\Omega=\Omega_{\mathrm{c}} \cup \Omega_{\mathrm{c}}^C \subset \mathbb{R}^3$ (see Figure
\ref{domains}) and $t \in \mathcal{I}=(t_0,t_\mathrm{end}] \subset \mathbb{R}$,
the evolution of electromagnetic fields is governed by the following Maxwell's
equations on $\Omega \times \mathcal{I}$, i.e.,
\begin{equation*}
  \Curl[]{\bh}  = \bj +\partial_t \bd,  \quad \Curl[]{\be}  = -\partial_t \bb , 
  \quad \Div[]{\bd} = \rho,  \quad \Div[]{\bb}  = 0,
  \label{eq:maxwell-equations}
\end{equation*}
and the constitutive laws, e.g. \cite{jackson-em-98}\refstepcounter{equation}\label{eq:const-laws-all}
\begin{equation}
    \bj(\bx, t)=\boldsymbol{\mathcal{J}}\big(\be(\bx, t),\bx\big), 
    \quad \bd(\bx, t)=\boldsymbol{\mathcal{D}}\big(\be(\bx, t),\bx\big), 
    \quad \bh(\bx, t)=\bH\big(\bb(\bx, t),\bx\big).
    \tag{{\theequation} a-c}
    \label{eq:const-laws-1}
\end{equation}\begin{figure}[t]
\centering
\scalebox{0.75}{\begingroup  \makeatletter  \providecommand\color[2][]{    \errmessage{(Inkscape) Color is used for the text in Inkscape, but the package 'color.sty' is not loaded}    \renewcommand\color[2][]{}  }  \providecommand\transparent[1]{    \errmessage{(Inkscape) Transparency is used (non-zero) for the text in Inkscape, but the package 'transparent.sty' is not loaded}    \renewcommand\transparent[1]{}  }  \providecommand\rotatebox[2]{#2}  \ifx\svgwidth\undefined    \setlength{\unitlength}{399.36533896bp}    \ifx\svgscale\undefined      \relax    \else      \setlength{\unitlength}{\unitlength * \real{\svgscale}}    \fi  \else    \setlength{\unitlength}{\svgwidth}  \fi  \global\let\svgwidth\undefined  \global\let\svgscale\undefined  \makeatother  \begin{picture}(1,0.64241587)    \put(0,0){\includegraphics[width=\unitlength,page=1]{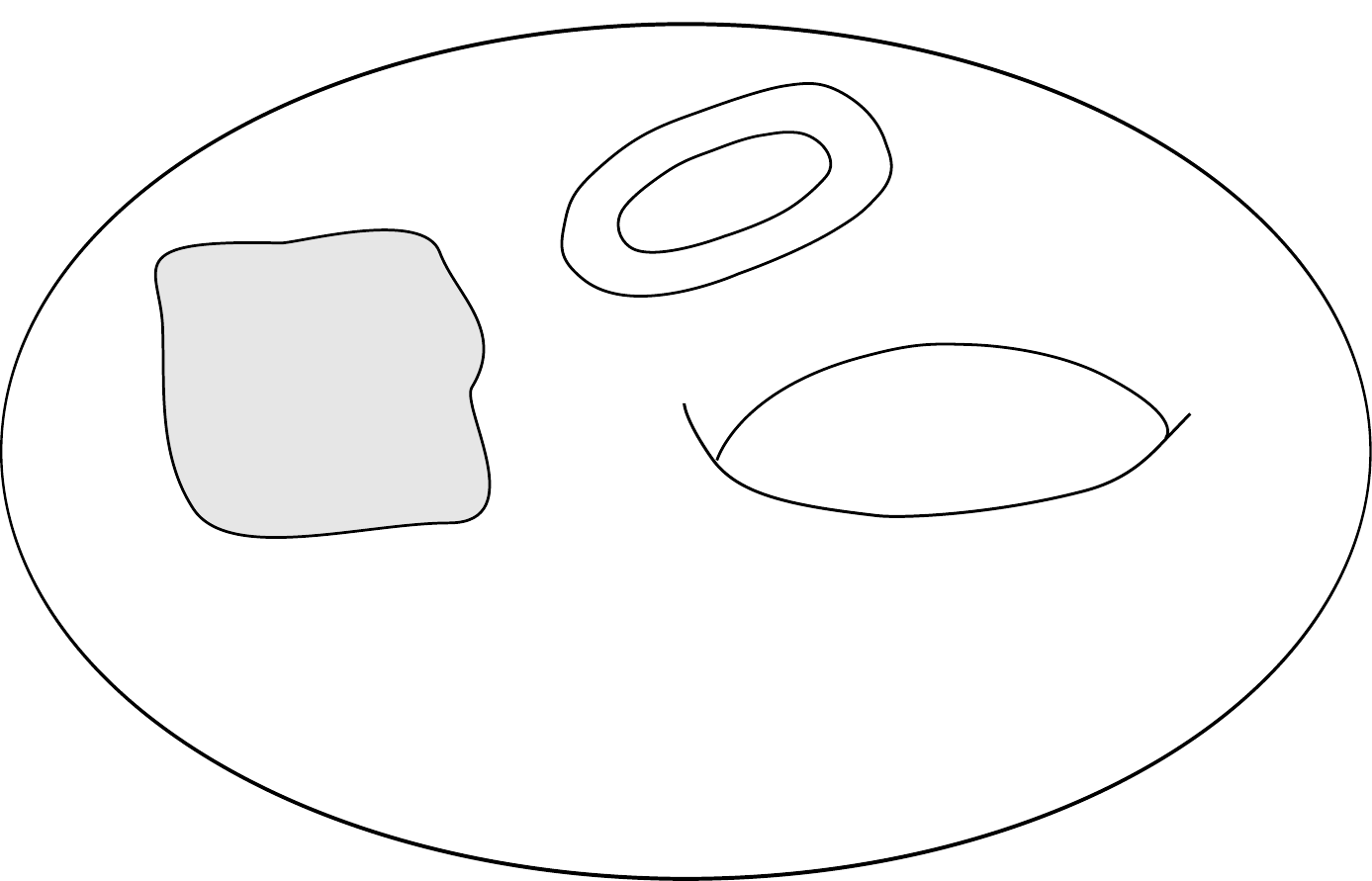}}    \put(1.77356528,2.89992387){\color[rgb]{0,0,0}\makebox(0,0)[lb]{\smash{}}}    \put(0.21020964,0.47840259){\color[rgb]{0,0,0}\makebox(0,0)[lb]{\smash{\textbf{$\Omega_c$}}}}        \put(0.21132869,0.3249628){\color[rgb]{0,0,0}\makebox(0,0)[lb]{\smash{\textbf{$\bj$}}}}    \put(0.41831252,0.3179366){\color[rgb]{0,0,0}\makebox(0,0)[lb]{\smash{\textbf{$\Omega_c^{C}$}}}}    \put(0.53899488,0.55058282){\color[rgb]{0,0,0}\makebox(0,0)[lb]{\smash{\textbf{$\Omega_s$}}}}        \put(0.46232225,0.63365788){\color[rgb]{0,0,0}\makebox(0,0)[lb]{\smash{\textbf{$\Gamma = \Gamma_h \cup \Gamma_e$}}}}    \put(-0.62268608,2.61004806){\color[rgb]{0,0,0}\makebox(0,0)[lt]{\begin{minipage}{4.09618556\unitlength}\raggedright \end{minipage}}}  \end{picture}\endgroup }
\caption{Bounded domain $\Omega$ and its subregions \cite{niyonzima-13,niyonzima-14}. 
The domain $\Omega$ can be split into the conductiong region $\Omega_{\mathrm{c}}$ (with $\sigma > 0$)  
and the non-conducting region $\Omega_{\mathrm{c}}^C = \Omega \backslash \Omega_{\mathrm{c}}$ (with $\sigma = 0$) 
which contains inductors $\Omega_{\mathrm{s}}$ where the current density $\bj_{\mathrm{s}}$ is imposed.
The boundary of the domain $\Gamma$ is such that $\Gamma = \Gamma_{\mathrm{e}} \cup \Gamma_{\mathrm{h}}$
with $\Gamma_{\mathrm{e}} \cap \Gamma_{\mathrm{h}} = \emptyset$. The region $\Gamma_{\mathrm{e}}$ is the part of 
the boundary where the tangential trace of $\be$ (resp. the normal trace of $\bb$) 
is imposed and $\Gamma_{\mathrm{h}}$ is the part of the boundary where the tangential trace 
of $\bh$ (resp. the normal trace of $\bd$ or $\bj$) is imposed. 
}
\label{domains}
\end{figure}

In these equations, $\bh$ is the magnetic field [A/m], $\bb$ the magnetic flux
density [T], $\be$ the electric field [V/m], $\bd$ the electric flux density
[C/m$^2$], $\bj$ the electric current density [A/m$^2$], and $\rho$ the electric
charge density [C/m$^3$]. 
The domain $\Omega_{\mathrm{c}}$ contains conductors whereas the domain $\Omega_{\mathrm{c}}^C$ contains
insulators. Additionally, suitable initial conditions and boundary 
conditions must be imposed for the problem to be well posed.

In this paper we consider only the `magnetoquasistatic' (MQS) case; it is
derived from Maxwell's equations by neglecting the displacement currents with
respect to the eddy currents $\partial_t \bd \ll \bj$. This can be justified if
$L \ll \lambda$ and $\delta \simeq L$ with $L$, the characteristic length of the
system, $\lambda$ the wavelength of the exciting source and $\delta$ the skin
depth. A more rigorous analysis can be found in \cite{schmidt-08}.
The resulting eddy current problem can be defined by the following MQS
approximation of Maxwell's equations \cite{bossavit-cem-98}
\refstepcounter{equation}\label{eq:ampere-faraday-gauss}
\begin{equation}
    \begin{aligned}
        \Curl[]{\bh}  = \bj , \quad \Curl[]{\be}  = -\partial_t \bb, \quad 
        \Div[]{\bb} = 0, 
    \end{aligned}
    \tag{{\theequation} a-c}\label{eq:ampere-faraday-gauss_1}
\end{equation}
and the relevant constitutive laws 
for $\bj$ (see equation (\ref{eq:const-laws-all} a)) 
and $\bh$ (see equation (\ref{eq:const-laws-all} b)).
For the applications treated in this paper the first (electric) constitutive law
will be considered of the form
$\bj(\bx, t)=\sigma(\bx) \, \be(\bx, t) + \bj_{\mathrm{s}}(\bx, t)$, with $\sigma$ the
(anisotropic) electric conductivity and $\bj_{\mathrm{s}}$ an imposed (source) electric
current density in $\Omega_{\mathrm{s}} \subset \Omega_{\mathrm{c}}^C$ [A/m$^2$]. The second (magnetic)
constitutive law can be linear, nonlinear reversible or nonlinear irreversible
(i.e. with hysteresis). Typical nonlinear reversible models include Brauer's model 
\cite{Brauer_1981aa}, Roug\'{e}'s formula \cite{rouge-mag-36} or splines.

Boundary conditions on the tangential component of the magnetic field (or on the
normal component of $\bj$) and on the normal component of the magnetic flux
density (or on the tangential component of $\be$) are imposed on complementary
parts $\Gamma_{\mathrm{h}}$ and $\Gamma_{\mathrm{e}}$ of the boundary $\Gamma = \partial \Omega =
\Gamma_{\mathrm{e}} \cup \Gamma_{\mathrm{h}}$:
\refstepcounter{equation}
\begin{equation}
    \bn \times \bh|_{\Gamma_{\mathrm{h}}} = \bh_t, \quad \bn \cdot \bb|_{\Gamma_{\mathrm{e}}} = b_n.
    \tag{{\theequation} a-b}\label{eq:bcs}
\end{equation}
In this paper, we use the modified vector potential formulation and write
$\bb$ and $\be$ as:
\refstepcounter{equation}\label{eq:vectpot}
\begin{equation}
    \bb = \Curl \ba \, , \quad	 
    \be = -\partial_t \ba
    \tag{{\theequation} a-b}\label{eq:vectpot_1}	
\end{equation}
with $\ba$ the magnetic vector potential [V\,s/m]. Therefore the essential boundary
condition $\bn \cdot \bb|_{\Gamma_\mathrm{e}}=\bn \cdot (\Curl[]{\ba})|_{\Gamma_\mathrm{e}}=b_n$
leads to the cancellation of the normal component of the $\CurlSymb$ and can be
fulfilled by imposing $\bn \times \ba|_{\Gamma_{\mathrm{e}}} = \boldsymbol{0}$.

The MQS problem \eqref{eq:ampere-faraday-gauss_1} together with the constitutive 
laws (\ref{eq:const-laws-all} a) and (\ref{eq:const-laws-all} c) can be solved
using the finite element method. To do this, a Galerkin formulation of the 
problem must be developed. Existence of the (weak) solutions presupposes some 
regularity assumptions on the data of the problem.
The conductivity $\sigma$ is defined such that the mapping $\bJ$ defined 
in (\ref{eq:const-laws-all} a) is monotone, nondecreasing and 
continuous in $\be$. For the linear electric law used in this paper, these conditions 
are fulfilled if $\sigma$ is bounded, i.e., if $\sigma \in L^{\infty}(\Omega)$. 
The mapping $\bH$ is assumed to be maximal monotone which presupposes that 
$\partial \bH/ \partial \bb$ is positive definite. This is the case for linear 
and nonlinear magnetic mappings but does not hold for hysteretic magnetic materials. 
The excitation term $\bj_{\mathrm{s}}$ needs also to be regular enough, e.g., 
$\bj_{\mathrm{s}} \in \Ltwo{(0,T); V(\Omega)^{\ast}}$. The space $V(\Omega) := \Hcurl[_{\mathrm{e}}]{\Omega}$ 
is the appropriate function space for the vector potential with boundary conditions 
on $\Gamma_{\mathrm{e}}$, and the superscript $^{\ast}$ is used
to denote its dual (see \cite{bossavit-cem-98, Bachinger_2005aa}).

Using these assumptions, the weak form of (\ref{eq:ampere-faraday-gauss}\,a) reads
\cite{Bachinger_2005aa, Abdulle_2015ab}: find $\ba \in \Ltwo{(0, T); V(\Omega) }$ 
with $\partial_t \ba \in \Ltwo{(0, T);$ $V(\Omega)^{\ast} }$
such that
\begin{equation}
    \ivol[_{\Omega_{\mathrm{c}}}]{ \sigma \partial_t \ba }{\ba'} 
     + \ivol[_{\Omega}  ]{ \bH(\Curl{\ba}) }{ \Curl{\ba'} } 
    \\
    + \isur[_{\Gamma_{\mathrm{h}}}]{ \bh_t}{\ba'} 
    - \ivol[_{\Omega_{\mathrm{s}}}]{ \bj_{\mathrm{s}} }{\ba'} = 0
  \label{eq:weakform_magdyn}
\end{equation}
holds for all test functions $\ba' \in V_0(\Omega)$, 
where the subscript $_0$ is used to denote homogeneous boundary Dirichlet conditions. 
More regularity in time and space for the solution can be 
obtained by imposing more regularity on the data of the problem \cite{Brezis_2010a}
Round brackets $(\cdot, \cdot)$ are used for volume integrals whereas angle brackets
$\left\langle\cdot, \cdot\right\rangle$ are used for surface integrals.  
The field $\ba$ derived from \eqref{eq:weakform_magdyn} must be gauged on $\Omega_{\mathrm{c}}^C$ 
to ensure its uniqueness. This can mathematically be achieved by factoring the 
space $\Hcurl[_{\mathrm{e}}^0]{\Omega}$ by gradients of scalar potentials, e.g., \cite{Bachinger_2005aa}.

\subsection{Multiscale}
Following \cite{niyonzima-13,niyonzima-14,visintin-08}, we use the subscript
$\varepsilon = l/L$ to denote quantities with rapid fluctuations. 
The length $l$ denotes the length of the periodic cell and the length $L$ denotes 
the characteristic length of the material or the minimum wavelength of the
exciting source current $\bj_{\mathrm{s}}(t)$. This wavelenght is defined as $\lambda_{\mathrm{min}} = c/f_{\mathrm{max}}$ 
where $c$ is the speed of light and $f_{\mathrm{max}}$ is the highest frequency 
obtained when $\bj_{\mathrm{s}}(t)$ is decomposed using the Fourier transform. The homogenized computational domain is assumed to be located far from the boundary 
$\Gamma$ such that the boundary term in \eqref{eq:weakform_magdyn} is independent of
$\varepsilon$.Using this convention  we can define the equivalent multiscale weak form for equation \eqref{eq:weakform_magdyn}.
\begin{prob}[Multiscale finescale problem]
Find $\ba^{\varepsilon} \in \Ltwo{(0, T); V(\Omega) }$ with
$\partial_t \ba^{\varepsilon} \in \Ltwo{(0, T); V(\Omega)^{\ast} }$
such that
\begin{equation}
    \ivol[_{\Omega_{\mathrm{c}}}]{ \sigma^{\varepsilon} \partial_t \ba^{\varepsilon} }{\ba^{\prime \, \varepsilon}} 
    + \ivol[_{\Omega}  ]{ \bH^{\varepsilon}(\Curl{\ba^{\varepsilon}}) }{ \Curl{\ba^{\prime \, \varepsilon}} }\\
    + \isur[_{\Gamma_{\mathrm{h}}}]{ \bh_t }{\ba^{\prime \, \varepsilon}} 
    - \ivol[_{\Omega_{\mathrm{s}}}]{ \bj_{\mathrm{s}} }{\ba^{\prime \, \varepsilon}} = 0
  \label{eq:weakform_magdyn_multiscale}
\end{equation}
holds for all test functions $\ba^{\prime \, \varepsilon} \in V_0(\Omega)$. 
\label{eq:multiscale_problem}
\end{prob}
\noindent This `finescale' weak form is used as the reference solution for
problems involving multiscale materials. The conductivity
$\sigma^{\varepsilon}$ and the material mapping $\bH^{\varepsilon}$ are 
defined by
\begin{equation}
  \sigma^{\varepsilon}(\bx) = \sigma\left(\frac{\bx}{\varepsilon}\right)  \hspace{5mm} \mathrm{and} \hspace{5mm} \bH^{\varepsilon}(\bb^{\varepsilon}(\bx, t), \bx) = \overline{\bH}(\bb^{\varepsilon}(\bx, t), \bx, \bx/\varepsilon)
\label{eq:multiscale_ml}
\end{equation} 
for all $\bx= (x_1, x_2, x_3)$ in $\Omega_H$, where $\Omega_H$ is the multiscale
computational domain.  The mapping $\overline{\bH}$ is used to represent
two-scale composite materials for which the characteristic length at the
mesoscale is $\varepsilon$ \cite{visintin-08}. By abuse of notation, we use
$\bH$ instead of $\overline{\bH}$ in the rest of the text.  In
\eqref{eq:multiscale_ml}, slow variations of the material law are accounted for
by the term $\bx$ while the fast fluctuations are accounted for by the term
$\bx/\varepsilon$ for $\varepsilon \ll 1$ (see
\cite{bensoussan-ahm-11,visintin-tsh-06,visintin-08}).

As an illustration, consider a two-dimensional linear magnetic material law $\bH(\bb(\bx, t), \bx) = \mu(\bx) \, \bb(\bx, t)$ 
with the magnetic permeability defined as 
\begin{equation*}
  \mu(\bx) = 
  \begin{cases} 
  \mu_1 & \quad \mathrm{if } (x_1  \mod \hspace{2mm} T) \leq \lambda,
  \\ 
  \mu_2 & \quad \mathrm{if } (x_1 \mod \hspace{2mm} T) > \lambda,
  \end{cases} 
  \label{eq:permeability}
\end{equation*}
for all $\bx \in \Omega_H$, positive $\mu_1, \mu_2$ and $\lambda < T$. 
The magnetic permeability $\mu$ is periodic with period $T$ and is 
representative of a stack of laminations made of two materials.
The division by the parameter $\varepsilon$ allows to make the period smaller 
(from $T$ to $\varepsilon \, T$). 
In the previous case, the permeability becomes: 
\begin{equation*}
  \mu(\bx/\varepsilon) = 
  \begin{cases} 
  \mu_1 & \quad \mathrm{if } (\bx  \mod \hspace{2mm} \varepsilon T) \leq \lambda,
  \\ 
  \mu_2 & \quad \mathrm{if } (\bx \mod \hspace{2mm} \varepsilon T) > \lambda,
  \end{cases} 
\end{equation*}
and is $\varepsilon$-periodic. Therefore the material law 
$\bH^{\varepsilon}(\bb^{\varepsilon}(\bx, t), \bx, \bx/{\varepsilon}) 
= \mu(\bx/{\varepsilon}) \, \bb^{\varepsilon}(\bx, t)$ is rapidly fluctuating for 
$\varepsilon \ll 1$.

In the following sections, the indices $\mathrm{M}, \mathrm{m}$ and $\mathrm{c}$
are used for denoting the macroscale, the mesoscale and correction terms, respectively.
The variables $\bx \in \Omega$ and $\by= (y_1, y_2, y_3)\in \Omega_{\mathrm{m}}$ are the
macroscale and the mesoscale coordinates and the mesoscale coordinates are only 
defined on the cell domain with the origin at the barycenter. 

\section{The heterogeneous multiscale method}
\label{sec:classical-hmm}
{Developments of this section are derived along the lines of ideas in 
\cite{niyonzima-13,niyonzima-14}.
The resolution of the finescale reference Problem \ref{eq:multiscale_problem} 
is computationally expensive for small values of $\varepsilon$. Hence, multiscale 
methods as HMM \cite{e-hmm-03} are necessary to reduce the computational costs 
and eventually make realistic simulations feasible.

For the MQS problems, HMM is based on the scale separation assumption
($\varepsilon \ll 1$) and was already illustrated in Figure \ref{FE2_ppe}.  When
HMM is applied, the finescale problem is replaced by one macroscale
problem defined on a coarse mesh covering the entire domain and accounting for
the slow variations of the finescale solution, and by many mesoscale problems 
defined on small, finely meshed areas around some points of interest of the
macroscale mesh (e.g. numerical quadrature points), and are used for computing
missing information.  The transfer of information between these problems is done
during the \textit{upscaling} and the \textit{downscaling} stages.

Equations governing the macroscale and the mesoscale are derived from the 
finescale problem using the asymptotic homogenization theories. The macroscale
problem is derived from the finescale problem using the classical weak convergence 
theory whereas the mesoscale problem is derived using the two-scale convergence 
\cite{visintin-08}. For these convergence theories to be applied, the solution of 
the finescale Problem \ref{eq:multiscale_problem} must exist and belong to appropriate 
function spaces (e.g.: reflexive or separable Banach spaces). This imposes some 
regularity conditions on $\sigma^{\varepsilon}$, $\bH^{\varepsilon}$ and on the 
excitation $\bj_{\mathrm{s}}$ which have been stated in Section \ref{sec:mqs-problem}.
However, the HMM has been numerically used for 
hysteretic magnetic laws that do not fulfill the monotonicity assumptions on 
$\bH^{\varepsilon}$ \cite{niyonzima-14}.

\subsection{The macroscale problem}
\noindent The macroscale weak form of the problem in $\ba$-formulation can be derived from 
equation \eqref{eq:weakform_magdyn} as follows~\cite{niyonzima-14}:
\begin{prob}[Macroscale weak problem]
Find $\ba_\mathrm{M} \in \Ltwo{(0, T); V(\Omega) }$ with 
$\partial_t \ba_\mathrm{M} \in \Ltwo{(0, T); V(\Omega)^{\ast} }$
such that

\begin{equation}
    \Big( \sigma_\mathrm{M} \partial_t \ba_\mathrm{M}, \ba_\mathrm{M}^{\prime}\Big)_{\Omega_\mathrm{c}} +
    \Big( \bH_\mathrm{M}(\bb_\mathrm{M}, \bx),\Curl[_\mathrm{x}]{\ba_\mathrm{M}^{\prime}} \Big)_{\Omega}\\ 
    + \Big< \bn \times \bh_\mathrm{M}, \ba_\mathrm{M}^{\prime} \Big>_{\Gamma_{\mathrm{h}}} 
    - \Big( \bj_\mathrm{s}, \ba_\mathrm{M}^{\prime} \Big)_{\Omega_\mathrm{s}} = 0
  \label{eq:classical_macro}
\end{equation}
holds for all 
$\ba_\mathrm{M}^{\prime} \in V_0(\Omega)$, 
$\bb_\mathrm{M} := \Curl[_{\mathrm{x}}]\ba_\mathrm{M}$
and $t \in \mathcal{I}=(t_0,t_\mathrm{end}]$.
\label{eq:classical_macroscale_problem}
\end{prob}
\noindent 
Thanks to the linearity of the electric law, the macroscopic conductivity
$\sigma_\mathrm{M}$ is obtained using the asymptotic expansion method \cite{bensoussan-ahm-11}
\begin{equation}
    \left(\sigma_\mathrm{M} \right)_{i\,j} = \frac{1}{|\Omega_{\mathrm{m}}|} \int_{\Omega_{\mathrm{m}}} \left( \sigma_{i\,j} 
    - \sum_{k} \sigma_{i\,k} \frac{\partial \chi^{j}(\by)}{\partial y_k} \right) \mathrm{d}y,
  \label{eq:Homogenized-Law-FE-HMM-Sigma}
\end{equation}
where $\chi^{j}$ is obtained by solving the cell problem: 
find $\chi^{j} \in V_G$ such that
\begin{equation}
    \Big( (\Grad[_y]{\psi})^T , \sigma (\Grad[_y]{\chi^{j}} - \be_j) \Big)_{\Omega_{\mathrm{m}}} = 0, 
    \quad \forall \psi \in V_G.
      \label{eq:CellProblem-Sigma}
\end{equation}
The space $V_G$ is the space $\Hone[]{\Omega_{\mathrm{m}}}$ with 
periodic boundary conditions while $\be_j$ is the unit vector in the 
$j^{\mathrm{th}}$ spatial direction.
\noindent The field $\bb_{\mathrm{c}} = \Curl[_\mathrm{y}]{\ba_{\mathrm{c}}}$
is the magnetic correction field obtained by solving the mesoscale 
problem corresponding to points $\bx \in \Omega$. The mesoscale fields depend
on the associated macroscale field and vice versa due to coupling of scales, i.e.,
\begin{equation}
  \ba_{\mathrm{c}}(\bx, \by,t) = \boldsymbol{\mathcal{A}}_{\mathrm{c}}\big(\by, \ba_{\mathrm{M}}(\bx,t)\big)
  \label{eq:bc-fe-hmm}
\end{equation}
where $\boldsymbol{\mathcal{A}}_\mathrm{c}$ denotes the solution operator 
of the mesoscale problem to be described in the following section \ref{fe-hmm-meso}. 
It is defined in analogy to the corrector operator defined  
in \cite{Henning_2015a} for nonlinear scalar elliptic problems.
The macroscopic magnetic law $\bH_\mathrm{M}$ in \eqref{eq:classical_macro} 
is computed at each point $(\bx, t) \in \Omega \times \mathcal{I}$ using both scales according to the averaging
formula from the two-scale convergence theory \cite{visintin-08}
\begin{equation}
\bH_\mathrm{M}(\bb_{\mathrm{M}}(\bx, t), \bx) 
:= \frac{1}{|\Omega_{\mathrm{m}}|} \int_{\Omega_{\mathrm{m}}} \bH\big(\bb_{\mathrm{c}}(\bx, \by, t) 
+ \bb_{\mathrm{M}}(\bx, t), \bx, \by\big)\;\mathrm{d} \by \quad \forall (\bx, t) \in \Omega \times \mathcal{I}
\label{eq:Homogenized-Law-FE-HMM}
\end{equation}
where $\bb_\mathrm{M} := \Curl[_{\mathrm{x}}]\ba_\mathrm{M}$ and 
$\bb_{\mathrm{c}} = \Curl[_\mathrm{y}]{\boldsymbol{\mathcal{A}}_{\mathrm{c}}(\by, \ba_{\mathrm{M}})}$. 
The underlying constitutive law $\bH$ is known for the heterogeneous phases at the
mesoscale level. 

\subsection{Mesoscale problems}
\label{fe-hmm-meso}
\noindent The following weak form of the mesoscale is defined from equation 
\eqref{eq:weakform_magdyn}  \cite{visintin-08,niyonzima-14}
\begin{prob}[Mesoscale weak problem]
Find the mesoscale correction $\ba_\mathrm{c} \in \Ltwo{(0, T); V_{\mathrm{per}}(\Omega_m) }$ with 
$\partial_t \ba_\mathrm{c} \in \Ltwo{(0, T); V_{\mathrm{per}}(\Omega_m)^{\ast} }$
using periodic boundary conditions such that 
\begin{equation}
  \Big(\sigma \partial_{t} \ba_{\mathrm{c}}, \ba_\mathrm{c}^{\prime}\Big)_{\Omega_\mathrm{mc}}
  + \Big( \bH(\bb_{\mathrm{c}} 
  + \bb_\mathrm{M}, \bx, \by), \Curl[_\mathrm{y}]{\ba_\mathrm{c}^{\prime}} \Big)_{ \Omega_\mathrm{m}} 
  \\
  - \Big( \sigma \be_\mathrm{M}, \ba_\mathrm{c}^{\prime} \Big)_{\Omega_\mathrm{mc}} = 0,
  \label{eq:classical_meso}
\end{equation}
for all $\ba_\mathrm{c}^{\prime} \in V_{\mathrm{per}}(\Omega_m)$,
magnetic correction field $\bb_{\mathrm{c}} = \Curl[_\mathrm{y}]{\ba_{\mathrm{c}}}$, 
and periodic boundary conditions and $\bb_{\mathrm{M}}$ given. 
The subscript $_{\mathrm{per}}$ is used to denote the use of periodic boundary
conditions. 
\label{eq:classical_mesoscale_problem}
\end{prob}
\noindent The macroscale magnetic and electric fields are defined as
\begin{align*}
  \bb_\mathrm{M}(\bx, t)&:=\Curl[_{\mathrm{x}}]\ba_\mathrm{M}(\bx,t)\\
  \be_\mathrm{M}(\bx, t)&:= -\partial_t\ba_\mathrm{M}(\bx,t)-\kappa (\partial_t\bb_\mathrm{M} \times \by)^{T}
\end{align*}
with $\kappa = 1$ for two-dimensional problems and $\kappa = 1/2$ for three-dimensional problems.
Existence and uniqueness of the mesoscale correction $\ba_{\mathrm{c}}$ motivates the introduction of 
the solution operator in \eqref{eq:bc-fe-hmm}. It can be formally deduced based on standard theory for nonlinear 
elliptic-parabolic problems, e.g. \cite{Bachinger_2005aa,barbu1976nonlinear,brezis1973operateurs,browder1970existence}. 

\subsection{Space and time discretization}
\label{fe-hmm-discrete}
Macroscale and mesoscale equations are solved using the finite element method. The 
first step consists in discretizing the computational domain into elements. 
The fields $\ba_{\mathrm{M}}^{H}$ and $\ba_{\mathrm{c}}^{h}$ are approximation 
of the continuous fields $\ba_{\mathrm{M}}$ and $\ba_{\mathrm{c}}$ on the 
discretized computational domain and  
$\ba_{\mathrm{M}}^{H} \in (0,T] \times \bW_{H,0}^\mathrm{M}$ and 
$\ba_{\mathrm{c}}^{h} \in (0,T] \times \bW_{h,0}^\mathrm{m}$ where
$\bW_{H,0}^\mathrm{M}$ and $\bW_{h,0}^\mathrm{m}$ are $H$(curl)-conforming edge 
finite elements.
In this paper we consider only lowest order $H$(curl)-conforming edge finite elements as test and basis functions, i.e., 
\begin{equation}
    \begin{aligned}
    & \bW_{H,0}^\mathrm{M} := \Big\{\bv \in \Hcurl[]{\Omega} \, \big| \, \bv \in 
    \mathcal{N}_0^{I}(K_M) \, \forall K_M \in \mathcal{T}_{H}^\mathrm{M} \Big\},
    \\
    & \bW_{h,0}^\mathrm{m} := \Big\{\bv \in \Hcurl[]{\Omega_\mathrm{m}} \, \big| \, \bv \in 
    \mathcal{N}_0^{I}(K_m), \, \forall K_\mathrm{m} \in \mathcal{T}_{\mathrm{h}}^\mathrm{m} \Big\}
    \end{aligned}
    \label{eq:fem_discrete_spaces}
\end{equation}
where $\mathcal{N}_0^{I}(K):= \big\{ \ba + \bb \times \bx\,|\,\ba,\bb \in \mathbb{R}^3 \big\}$, 
see e.g. \cite{bossavit-cem-98}. The triangulations $\mathcal{T}_{H}^{\mathrm{M}}$ and $\mathcal{T}_{\mathrm{h}}^{\mathrm{m}}$ are 
defined on the macroscale and mesoscale domains, respectively.

The weak forms \eqref{eq:classical_macro} and \eqref{eq:classical_meso} can then 
be computed using numerical quadrature rules. This implies that the quantities 
involved in the integrations (e.g. the homogenized material law) be known at Gau{\ss} points ($i=1,\ldots,N_\mathrm{GP}$).
Omitting the superscript $^i$ used for the numbering of the Gau{\ss} in 
the approximation of the mesoscale field $\ba_{\mathrm{c}}^{\mathrm{h}}(\by,t)$,
the discrete spaces give rise to the approximations
\begin{equation}
  \ba_{\mathrm{M}}(\bx,t) \approx \ba_{\mathrm{M}}^{H}(\bx,t) = \sum_{p=1}^{N_\mathrm{M}} \ah_{\mathrm{M},p}(t) \ba_{\mathrm{M},p}(\bx)
  \qquad\mathrm{and}\qquad
  \ba_{\mathrm{c}}^{(i)}(\by,t) \approx \ba_{\mathrm{c}}^{h}(\by,t) = \sum_{p=1}^{N_\mathrm{c}} \ah_{\mathrm{c},p}(t) \ba_{\mathrm{c},p}(\by).
  \label{eq:fem_meso}
\end{equation}
Testing \eqref{eq:classical_macro} and \eqref{eq:classical_meso} yields
the macroscale mass matrix
\begin{align}
    \nonumber
    \bM_{\mathrm{M}}&:=\Big( \sigma_\mathrm{M} \ba_\mathrm{M}, \ba_\mathrm{M}^{\prime}\Big)_{\Omega_\mathrm{c}}
\intertext{which is singular due to $\sigma_M=0$ on $\Omega_{\mathrm{c}}^C$ and the stiffness term}
    \bF_{\mathrm{M}}(\bah_\mathrm{M},\bah_{\mathrm{c}})&:=\Big( \bH_\mathrm{M}\big(\Curl[_\mathrm{y}]{\ba_{\mathrm{c}}} + \Curl[_{\mathrm{x}}]{\ba}_\mathrm{M},\bx \big),\Curl[_\mathrm{x}]{\ba_\mathrm{M}^{\prime}} \Big)_{\Omega}
    + \Big< \bh_t, \ba_\mathrm{M}^{\prime} \Big>_{\Gamma_{\mathrm{h}}} 
    - \Big( \bj_\mathrm{s}, \ba_\mathrm{M}^{\prime} \Big)_{\Omega_\mathrm{s}}
    \label{eq:meso_stiffness}
\end{align}
with $\ba_{\mathrm{c}}=[\ba_{\mathrm{c}}^{(1)},\ldots,\ba_{\mathrm{c}}^{(N_\mathrm{GP})}]$. Similar definitions hold for 
$\bM_{\mathrm{m}}$ and $\bF_{\mathrm{m}}$ on the mesoscale.
The extension to higher order edge elements or nodal elements for 2D problems is straightforward.
Following the classical approach, numerical quadrature rules are used to compute the weak forms. For the 
macroscale problem, we use numerical quadrature with one Gau{\ss} point which 
is enough to capture the slow variations of the missing material law at Gau{\ss} points. 
The missing material law can also be computed at the barycenter of the element \cite{abdulle2012coupling}.

\begin{prob}[Semidiscrete multiscale problem]
Find waveforms $[\bah_{\mathrm{M}}(t),\bah_{\mathrm{c}}^{(1)}(t), \ldots, \bah_{\mathrm{c}}^{(N_\mathrm{GP})}(t)]$
such that
\begin{equation}
    \bM_{\mathrm{M}}\partial_t \bah_\mathrm{M}+\bF_{\mathrm{M}}(\bah_\mathrm{M},\bah_{\mathrm{c}})=0, 
  \label{eq:newton-macro-semidiscrete}
\end{equation}
and for the mesoscale problems $i=1,\ldots,N_\mathrm{GP}$
\begin{equation}
    \bM_{\mathrm{m}}\partial_t \bah_{\mathrm{c}}^{(i)}+\bF_\mathrm{m}(\bah_{\mathrm{c}}^{(i)},\bah_{\mathrm{M}}^{(i)}, \partial_t\bah_\mathrm{M}^{(i)})=0
    \label{eq:newton-meso-semidiscrete}
\end{equation}
for a given set of initial values $[\bah_{\mathrm{M}}(t_0),\bah_{\mathrm{c}}^{(1)}(t_0), \ldots, \bah_{\mathrm{c}}^{(N_\mathrm{GP})}(t_0)]$.
\end{prob}

Finally, the time-dependent Problem (\ref{eq:newton-macro-semidiscrete}-\ref{eq:newton-meso-semidiscrete}) can be solved using any classical (implicit) 
time integration scheme followed by a nonlinear solution method. 
In the simplest case, i.e., using the backward Euler scheme, the following 
nonlinear problem has to be solved for
for each time step:
\begin{prob}[Nonlinear, discrete multiscale problem]
Find the solutions
\begin{equation*}
  [\bah_{\mathrm{M}}^{(k)},\bah_{\mathrm{c}}^{(1, k)}, \ldots, \bah_{\mathrm{c}}^{(N_\mathrm{GP}, k)}]\in \mathrm{I\!R}^{N_\mathrm{M}+N_\mathrm{GP}\cdot N_\mathrm{c}}
\end{equation*}
such that
\begin{equation}
	\bR_{\mathrm{M}}\left(\bah_{\mathrm{M}}^{(k)},\bah_{\mathrm{c}}^{(i, k)}\right):=\bM_{\mathrm{M}}
	\frac{\bah_{\mathrm{M}}^{(k)} - \bah_{\mathrm{M}}^{(k-1)}}{\Delta t_k}
	+
	\bF_{\mathrm{M}} \Big(\bah_{\mathrm{M}}^{(k)}, \bah_{\mathrm{c}}^{(k)} \Big)
	= 0, 
  \label{eq:newton-macro}
\end{equation}
and for the mesoscale problems $i=1,\ldots,N_\mathrm{GP}$
\begin{equation}
    \bR_{\mathrm{m}}\left(\bah_{\mathrm{M}}^{(k)},\bah_{\mathrm{c}}^{(i, k)}\right):=
    \bM_{\mathrm{m}}
    \frac{\bah_{\mathrm{c}}^{(i, k)} - \bah_{\mathrm{c}}^{(i, k-1)}}{\Delta t_k}
    +
    \bF_{\mathrm{m}} \left(\bah_{\mathrm{c}}^{(i, k)}, \bah_{\mathrm{M}}^{(i, k)},\\
    \frac{\bah_{\mathrm{M}}^{(i, k)} - \bah_{\mathrm{M}}^{(i, k-1)}}{\Delta t_k} \right) 
    = 0,
    \label{eq:newton-meso}
\end{equation}
where the superscript $k$ is used to denote the approximations at time
instants $t_{k}\in[t_0,t_\mathrm{end}]$ , e.g. 
$\bah_{\mathrm{M}}^{(k)}\approx\bah_{\mathrm{M}}(t_k)$ and $\Delta t_k:=t_{k+1}-t_{k}$ is the corresponding time step size.
\label{eq:discret_multiscale_problem}
\end{prob}
\noindent The following loops are defined for the monolithic and the waveform relaxation 
methods: the loop for the number of time windows (TW) with $1 \leq n \leq N_{\mathrm{TW}}$, 
the loop for the number of waveform relaxation iterations
(WR) with $1 \leq l \leq N_{\mathrm{WR}}$, the loop for the number of time stepping (TS) with
$1 \leq k \leq N_{\mathrm{TS}}$, the loop for the number of Newton--Raphson nonlinear
iterations (NR) with $1 \leq j \leq N_{\mathrm{NR}}$ and the loop for the number of Gau{\ss}
points (GP) with $1 \leq i \leq N_{\mathrm{GP}}$. Table \ref{tab:loops}
summarizes the loops, the letter used for indexing them and the total number of
iterations for each loop.
\begin{table}[ht!]
\caption{Loops involved in the monolithic and waveform relaxation HMM.}
\label{tab:loops}
\medskip\centering
\begin{tabular}{| l | c | c | c |}
\hline
Type of loop        & Loop   & Index   & Maximum number of iterations\\
\hline
Time window         & TW     & $n$     & $N_{\mathrm{TW}}$   \\
Waveform relaxation & WR     & $l$     & $N_{\mathrm{WR}}$   \\
Time stepping       & TS     & $k$     & $N_{\mathrm{TS}}$   \\
Newton--Raphson     & NR     & $j$     & $N_{\mathrm{NR}}$   \\
Gau{\ss} points     & GP     & $i$     & $N_{\mathrm{GP}}$   \\
\hline
\end{tabular}
\end{table}

\section{Monolithic HMM}
\label{fe-hmm-mono}
In the following a rigorous interpretation of the time-stepping procedures proposed in the context of HMM is given in terms of Problem~\ref{eq:discret_multiscale_problem}. These derivations are an important building block for the comparison with the waveform relaxation approach in Section~\ref{sec:waveform}. 

In \cite{niyonzima-14} the Algorithm~\ref{alg:FE-HMM-Macro} was proposed. 
For each time step, a nonlinear system on the macroscale is solved using the Newton--Raphson
method until convergence is reached. In each Newton iteration the material law \eqref{eq:Homogenized-Law-FE-HMM} is evaluated
$$
\bH_\mathrm{M}^{(i, j, k)}
:= \frac{1}{|\Omega_{\mathrm{m}}|} \int_{\Omega_{\mathrm{m}}} \bH\Big(
  \bb_\mathrm{c}^{(i, j, k)},
  \Curl[_{\mathrm{x}}]\ba_\mathrm{M}^{(j, k)}, 
  \bx,\by
\Big)\;\mathrm{d} \by
$$
where $\bb_\mathrm{c}^{(i,j, k)} = \Curl[_\mathrm{y}]{\boldsymbol{\mathcal{A}}_{\mathrm{c}}(\by, \ba_{\mathrm{M}}^{(j, k)})}$ is obtained from the discretized version of the nonlinear solution operator 
given in \eqref{eq:bc-fe-hmm}. This is implemented by solving the nonlinear equation \eqref{eq:newton-meso} again by the Newton--Raphson method using $N_{\mathrm{NR}}^{\mathrm{m}}$ iterations, cf. Algorithm~\ref{alg:FE-HMM-Meso}. This relaxation within the Newton scheme corresponds to a monolithic time-stepping scheme 
although it features parallel evaluations at the Gau{\ss} points at the level of the nonlinear solver. 
The two nested new Newton loops (inner and outer) are a special case of a \textit{parallel multilevel Newton scheme} 
as they are used for example in circuit simulation \cite{Grab_1996aa}. 
Let us state the equivalence for the case in which only one inner iteration of a 
simplified Newton--Raphson scheme is carried out.
This is closely related to the Newton--Raphson scheme developed in 
\cite{Henning_2015a} which involves the evaluation of the Fr\'{e}chet derivative 
of the nonlinear corrector operator.
\medskip 

\noindent Let $\bah_{\mathrm{M}}^{(j,k)}$ and $\bah_{\mathrm{c}}^{(i,j,k)}$ denote the $j^{\mathrm{th}}$
Newton--Raphson iterates. Then we define 
\begin{equation}
\bF_{\mathrm{M}}^{(j, k)} := \bF_{\mathrm{M}} \Big(\bah_{\mathrm{M}}^{(j, k)}, 
\bah_{\mathrm{c}}^{(j, k)} \Big) 
\quad , \quad
\bF_{\mathrm{m}}^{(i, j, k)} := \bF_{\mathrm{m}} \left(\bah_{\mathrm{c}}^{(i, j, k)}, \bah_{\mathrm{M}}^{(i, j, k)},
\frac{\bah_{\mathrm{M}}^{(i, j, k)} - \bah_{\mathrm{M}}^{(i, j, k-1)}}{\Delta t_k} \right).
\end{equation}

\begin{prop}
Solving the monolithic system \eqref{eq:newton-macro}-\eqref{eq:newton-meso} with the Newton--Raphson scheme using the Jacobian
\begin{equation}
J_{\bR}^{(j,k)} := 
\frac{1}{\Delta t_k}
\begin{pmatrix} 
    \displaystyle{\bM_{\mathrm{M}}\vphantom{\frac{\partial \bF^{(j,k)}_{\mathrm{M}}}{\partial \bah_{\mathrm{M}}^{(j,k)}}}}& 
    \displaystyle{0} & 
    \displaystyle{\cdots} & 
    \displaystyle{0} 
    \\ 
    \displaystyle{0} & 
    \displaystyle{\bM_{\mathrm{m}}\vphantom{\frac{\partial \bF^{(j,k)}_{\mathrm{M}}}{\partial \bah_{\mathrm{M}}^{(j,k)}}}} & 
    \displaystyle{0} & 
    \displaystyle{0}    
    \\ 
    \displaystyle{\vdots} & 0 & \displaystyle{\ddots} & 0 
    \\ 
    \displaystyle{0} & 
    0 & 
    0 & 
    \displaystyle{\bM_{\mathrm{m}}\vphantom{\frac{\partial \bF^{(j,k)}_{\mathrm{M}}}{\partial \bah_{\mathrm{M}}^{(j,k)}}}}
\end{pmatrix}+
\begin{pmatrix} 
    \displaystyle{\frac{\partial \bF^{(j,k)}_{\mathrm{M}}}{\partial \bah_{\mathrm{M}}^{(j,k)}}}& 
    \displaystyle{\frac{\partial \bF^{(j,k)}_{\mathrm{M}}}{\partial \bah_{\mathrm{c}}^{(1,j,k)}}} & 
    \displaystyle{\cdots} & 
    \displaystyle{\frac{\partial \bF^{(j,k)}_{\mathrm{M}}}{\partial \bah_{\mathrm{c}}^{(N_{\mathrm{GP}},j,k)}}} 
    \\ 
    \displaystyle{\frac{\partial \bF^{(1,j,k)}_{\mathrm{m}}}{\partial \bah_{\mathrm{M}}^{(j,k)}}} & 
    \displaystyle{\frac{\partial \bF^{(1,j,k)}_{\mathrm{m}}}{\partial \bah_{\mathrm{c}}^{(1,j,k)}}} & 
    \displaystyle{0} & 
    \displaystyle{0}    
    \\ 
    \displaystyle{\vdots} & 0 & \displaystyle{\ddots} & 0 
    \\ 
    \displaystyle{\frac{\partial \bF^{(N_{\mathrm{GP}},j,k)}_{\mathrm{m}}}{\partial \bah_{\mathrm{M}}^{(j,k)}}} & 
    0 & 
    0 & 
    \displaystyle{\frac{\partial \bF^{(N_{\mathrm{GP}},j,k)}_{\mathrm{m}}}{\partial \bah_{\mathrm{c}}^{(N_{\mathrm{GP}},j,k)}}}
\end{pmatrix}\label{eq:Jacobian-full}
\end{equation}is equivalent to the scheme proposed in Algorithms~\ref{alg:FE-HMM-Macro} and \ref{alg:FE-HMM-Meso} if 
\begin{itemize}
  \item[(a)] no inner Newton iterations on the mesoscale are carried out, i.e., 
  $N_{\mathrm{NR}}^{\mathrm{m}}=1$ and 
  \item[(b)] the sensitivity of the mesoscale Jacobian w.r.t. to the macroscale 
  is disregarded, i.e., $\partial J_\mathrm{m}^{(i, j, k)}/\partial\bah_{\mathrm{M}}^{(j,k)}=0$
\end{itemize}
\end{prop}

\begin{proof}
The equivalence is easily established by comparing the solution operator 
${\mathcal{A}}_{\mathrm{c}}$ as used in Algorithms~\ref{alg:FE-HMM-Macro}-\ref{alg:FE-HMM-Meso} 
to the Schur complement of the Jacobian $J_{\bR}^{(j,k)}$ as given in \eqref{eq:Jacobian-full} and 
already proposed in \cite{bottauscio-13} for the Variational Multiscale Method. The latter reads
\begin{equation}
\bar{J}_{\bR}^{(j,k)} :=
\frac{\bM_{\mathrm{M}}}{\Delta t_k}+
\displaystyle{\frac{\partial \bF^{(j,k)}_{\mathrm{M}}}{\partial \bah_{\mathrm{M}}^{(j,k)}}}
-\sum_{i=1}^{N_{\mathrm{GP}}}
\left(
  \displaystyle{\frac{\partial \bF^{(j,k)}_{\mathrm{M}}}{\partial \bah_{\mathrm{c}}^{(i,j,k)}}}  
  \left(
    \displaystyle{\frac{\bM_{\mathrm{m}}}{\Delta t_k}+\frac{\partial \bF^{(i,j,k)}_{\mathrm{m}}}{\partial \bah_{\mathrm{c}}^{(i,j,k)}}}
  \right)^{-1} \, 
  \displaystyle{\frac{\partial \bF^{(i,j,k)}_{\mathrm{m}}}{\partial \bah_{\mathrm{M}}^{(j,k)}}}
\right).
\label{eq:Jacobian-schur-complement-macro}
\end{equation}
Since Assumption (a), i.e., $N_{\mathrm{NR}}^{\mathrm{m}}=1$ in Algorithm~\ref{alg:FE-HMM-Meso} holds, the discretized version of the solution operator $\boldsymbol{\mathcal{A}}_{\mathrm{c}}$ applied to the linearized problem \eqref{eq:newton-meso} can be explicitly given as 
$$
\bah_{\mathrm{c}}^{(i,j+1,k)}=
\boldsymbol{\mathcal{A}}_{\mathrm{c}}^{(i)}(\bah_{\mathrm{M}}^{(j,k)})
=
\bah_{\mathrm{c}}^{(i,j,k)}
-\left(J_\mathrm{m}^{(i, j, k)}\right)^{-1}
\bR_{\mathrm{m}}\left(\bah_{\mathrm{M}}^{(j, k)},\bah_{\mathrm{c}}^{(i, j, k)}\right)
\qquad
\mathrm{with}
\qquad
J_\mathrm{m}^{(i, j, k)}:=\frac{\bM_{\mathrm{m}}}{\Delta t_k}+\frac{\partial \bF^{(i,j,k)}_{\mathrm{m}}}{\partial \bah_{\mathrm{c}}^{(i,j,k)}},
$$
which yields immediately the derivative with respect to the macro scale
\begin{equation}
  \label{eq:Homogenized-Law-FE-HMM-1}
\frac{\partial {\boldsymbol{\mathcal{A}}}_{\mathrm{c}}^{(i)}}{\partial \bah_{\mathrm{M}}^{(j,k)}}(\bah_{\mathrm{M}}^{(j,k)})
\approx
-\left(J_\mathrm{m}^{(i, j, k)}\right)^{-1}
\frac{\partial \bF^{(i,j,k)}_{\mathrm{m}}}{\partial \bah_{\mathrm{M}}^{(j,k)}}.
\end{equation}
where the contribution from the Jacobian $J_\mathrm{m}^{(i, j, k)}$ is disregarded due to Assumption (b). Summing up all contributions 
$$
\boldsymbol{\mathcal{A}}_{\mathrm{c}}(\bah_{\mathrm{M}}^{(j,k)})=
\sum_{i=1}^{N_{\mathrm{GP}}}\boldsymbol{\mathcal{A}}_{\mathrm{c}}^{(i)} (\bah_{\mathrm{M}}^{(j,k)}),
$$
plugging them into the macroscale stiffness matrix \eqref{eq:meso_stiffness} and exploiting \eqref{eq:Homogenized-Law-FE-HMM-1} concludes the proof 
$$
J_\mathrm{M}^{(j, k)}
=
\frac{\bM_{\mathrm{M}}}{\Delta t_k}
+
\frac{\mathrm{d}\bF_{\mathrm{M}}}{\mathrm{d} \bah_\mathrm{M}}\left(\bah_\mathrm{M}^{(j,k)},\boldsymbol{\mathcal{A}}_{\mathrm{c}}(\bah_{\mathrm{M}}^{(j,k)})\right)
\stackrel{!}{=}
\bar{J}_{\bF}^{(j,k)}
.
\qedhere
$$

\end{proof}

In practice the assumptions (a) and (b) can be violated and one will end up with a different variant of the Newton--Raphson scheme. For example: the computation \eqref{eq:Homogenized-Law-FE-HMM-1} involves the derivative of 
the correction terms $\boldsymbol{\mathcal{A}}_{\mathrm{c}}$ with respect to 
the macroscale magnetic density $\bah_{\mathrm{M}}$,
In \cite{niyonzima-14}, it was proposed to solve several mesoscale problems  
per Gau{\ss} point in parallel: for a two-dimensional problem, $N^\mathrm{m}_\mathrm{dim}=3$ problems were
solved to approximate the Jacobian, the first one with the nominal macroscale source 
(e.g. $\bah_{\mathrm{M}}^{(i)}$) and the other two with a perturbated source magnetic density 
$\bah_{\mathrm{M}}^{(i)} + \delta^{(i)} \bah_i$ where $\delta^{(i)}$ is a small 
perturbation and $\bah_i$ is a vector oriented along the $x$, $y$ or $z$-axes.  
Similarly, in 3D $N^\mathrm{m}_\mathrm{dim}=4$ problems need to be solved. 
On the other hand, one can use a fixed point iteration scheme as suggested 
in \cite{Abdulle_2015aa}, which is actually the limit case of a a waveform 
relaxation approach, where time window size equals time step size.
\begin{center}
\begin{algorithm}[!]
\SetKwInput{Input}{Input}
\SetKwInput{Output}{Output}
\LinesNumbered
\DontPrintSemicolon
\Input{macroscale source $\bj_{\mathrm{s}}$ and mesh.}
\Output{fields (macro/meso), global quantities.}
\Begin
{   $t \gets t_{0}$, initialize the macroscale field $\ba_{\mathrm{M}}|_{t_0} = \ba_{\mathrm{M}0}$,\\
    \emph{\# begin the macroscale time loop (index $k$)}\\
    \For{$(k \gets 1$ \KwTo $N_{\mathrm{TS}} )$}{ 
    \emph{\# begin the macroscale NR loop (index $j$)}
    \\ 
        \For{$(j \gets 1$ \KwTo $N_{\mathrm{NR}}^{\mathrm{M}} )$}{
        \emph{\# parallel resolution of mesoscale problems  (index $i$)}
        \\
        \For{$(i \gets 1$ \KwTo $N_{\mathrm{GP}} )$}{
            downscale the sources $\bah_{\mathrm{M}}^{(i, j, k-1)}$,
            \vspace{1mm}
            \\
            compute $\bah_{\mathrm{c}}^{(i, j, k)}=\boldsymbol{\mathcal{A}}_{\mathrm{c}}(\by,\bah_{\mathrm{M}}^{(i, j, k)})$,
            see Algorithm~\ref{alg:FE-HMM-Meso}
            \\
            compute the homogenized law $\bH_{\mathrm{M}}^{(i, j, k)}$ and $\partial 
            \boldsymbol{\mathcal{H}}_{\mathrm{M}}^{(i, j, k)}/ \partial \bb_{\mathrm{M}}^{(i, j, k)}$,
            \\
            \vspace{1mm}
            upscale the homogenized law $\bH_{\mathrm{M}}^{(i, j, k)}$ and $\partial 
            \boldsymbol{\mathcal{H}}_{\mathrm{M}}^{(i, j, k)}/ \partial \bb_{\mathrm{M}}^{(i, j, k)}$, 
        }
        assemble the Jacobian $\displaystyle{J_\mathrm{M}^{(j, k)}=
        \frac{1}{\Delta t_k}\bM_{\mathrm{M}}+\frac{\mathrm{d}\bF_{\mathrm{M}}}{\mathrm{d} 
        \bah_\mathrm{M}}\left(\bah_\mathrm{M}^{(j,k)},\boldsymbol{\mathcal{A}}_{\mathrm{c}}
        (\bah_{\mathrm{M}}^{(j,k)})\right)}$ to solve \eqref{eq:newton-macro},
        }
    }
}
\caption{Pseudocode for the monolithic FE-HMM}
\label{alg:FE-HMM-Macro}
\end{algorithm}
\end{center}
\begin{center}
\begin{algorithm}[!]
\SetKwInput{Input}{Input}
\SetKwInput{Output}{Output}
\LinesNumbered
\DontPrintSemicolon
\Input{macroscale sources $\bah_{\mathrm{M}}^{(i, j, k)}$ and the mesoscale mesh.}
\Output{homogenized law $\bH_{\mathrm{M}}^{(i, j, k)}$, per Gau{\ss} point for 
$N^\mathrm{m}_\mathrm{dim}$ problems.}
\Begin
{   
    prescribe periodic boundary conditions, impose sources, \\
    $t \gets t_{\mathrm{M}}$, initialize the correction $\ba_{\mathrm{c}}|_{t_{\mathrm{M}}}$,\\
    solve $N^\mathrm{m}_\mathrm{dim}$ mesoscale problems for the $k^{th}$ time step,\\ 
    \For{$(p \gets 1$ \KwTo $N^\mathrm{m}_\mathrm{dim})$}{
            \emph{\# begin the mesoscale NR loop (index $j$)}\\
            \For{$(j \gets 1$ \KwTo $ N_{\mathrm{NR}}^{\mathrm{m}} )$}{
                assemble the Jacobian $J_\mathrm{m}^{(i, j, k)}$ 
                to solve \eqref{eq:newton-meso}.
            }
    }
}
\caption{Pseudocode for one mesoscale problem}
\label{alg:FE-HMM-Meso}
\end{algorithm}
\end{center}
\section{The waveform relaxation method}
\label{sec:waveform}
Waveform relaxation methods solve time dependent problems iteratively, i.e., they
generalize the classical ideas of Gau{\ss}--Seidel and Jacobi iteration to the time
domain. The method starts with an initial guess of the solution over a time interval
and computes iteratively approximations of increasing accuracy \cite{Gander_2012aa}.
Typically the problem is decomposed into subproblems and each subproblem is solved separately. 
Let us consider the two ordinary differential equations 
\begin{align*}
	\partial_t y_1 &= f_1(y_1,y_2) \\
	\partial_t y_2 &= f_2(y_1,y_2). 
\end{align*}
A monolithic or strongly coupled approach discretizes the problem in time as one system of equations. On the other hand, an iterative Gau{\ss}--Seidel type scheme 
\begin{align*}
	\partial_t y_1^{(l)} &= f_1(y_1^{(l)},y_2^{(l-1)})\\
	\partial_t y_2^{(l)} &= f_2(y_1^{(l)},y_2^{(l)\phantom{-1}})
\end{align*}
will resolve both equations subsequently, e.g., the first one for the unknown $y_1^{(l)}(t)$ on $t\in\mathcal{I}$ while considering $y_2^{(l-1)}(t)$ on $t\in\mathcal{I}$ given and vice-versa. The very first iteration requires an initial guess $y_2^{(0)}$, which is typically obtained by constant extrapolation, \cite{Arnold_2001aa}. In the simplest case an implicit Euler method can be chosen for time stepping, e.g.
\begin{align*}
	\frac{y_1^{(k,l)}-y_1^{(k-1,l)}}{\Delta t_k} &= f_1(y_1^{(k,l)},y_2^{(k,l-1)})\\
	\frac{y_2^{(k,l)}-y_2^{(k-1,l)}}{\Delta t_k} &= f_2(y_1^{(k,l)},y_2^{(k,l)\phantom{-1}}).
\end{align*}
where $y_1^{(k,l)}$ describes the unknown $y_1$ at time $t_k$ and iteration $l$; $\Delta t_k$ denotes the $k$-th time step size for both problems. Obviously, the iteration scheme allows to combine different time integrators with independent time step sizes. It is therefore often referred to as co-simulation or weak coupling.
The convergence is well understood and unconditionally guaranteed for systems of ordinary differential equations \cite{Lelarasmee_1982ab,Burrage_1995aa}. However, already in the case of simple differential algebraic equations, e.g. the system
\begin{align*}
	\partial_t y_1^{(l)} &= f_1(y_1^{(l)},z_1^{(l)},y_2^{(l-1)},z_2^{(l-1)})\\
	0 &= g_1(y_1^{(l)},z_1^{(l)},y_2^{(l-1)},z_2^{(l-1)})
	\qquad\mathrm{with }
	\det\left(\frac{\partial g_1}{\partial z_1^{(l)}}\right)\neq0
	\\
	\partial_t y_2^{(l)} &= f_2(y_1^{(l)},z_1^{(l)},y_2^{(l)\phantom{-1}},z_2^{(l)\phantom{-1}})\\
	0 &= g_2(y_1^{(l)},z_1^{(l)},y_2^{(l)\phantom{-1}},z_2^{(l)\phantom{-1}})
	\qquad\mathrm{with }
	\det\left(\frac{\partial g_2}{\partial z_2^{(l)}}\right)\neq0
\end{align*}
the convergence of the fixed point iteration is conditional. 
In particular the dependence of algebraic equations on old algebraic iterates is 
critical, i.e., the Jacobian ${\partial g_1}/{\partial z_2^{(l)}}$ 
must be sufficiently small, \cite{Miekkala_1987aa}. The convergence for more complex problems,
possibly with higher DAE-index, is even more involved, \cite{Crow_1994}.

Waveform relaxation has been originally applied in the simulation of electrical networks 
but has been applied in various disciplines. Recently, the method was rediscovered to
cosimulate coupled problems \cite{Schops_2010aa}. The method converges particularly
fast on small intervals and hence it is common to subdivide the time interval of
interest into so called time windows and to apply the method on each time window
separately. This subdivision does not hinder the overall convergence since the error
propagation from windows to window can be controlled  \cite{Bartel_2013aa}.
Waveform relaxation is a particular parallel-in-time methods and hence closely 
linked to Parareal \cite{Lions_2001,Gander_2007} which has also recently been applied  
to multiscale problems \cite{Astorino_2015}.

\section{Waveform Relaxation HMM}
\label{sec:waveform-homogenization}

We employ a waveform relaxation-based approach with windowing \cite{white-85}. 
Weak forms similar to \eqref{eq:classical_macro} for the macroscale and 
\eqref{eq:classical_meso} for the mesoscale problem are solved on a series of 
time windows $\mathcal{I}_n=(t_{n-1}, t_{n}]\subset \mathcal{I}$ ($n=1,2,\ldots,N_{\mathrm{TW}}$). 
On each time window, macroscale and mesoscale problems are solved separately in 
time-domain, 
such that waveforms, e.g., $\bah_\mathrm{M}(t)$, are obtained. 
Afterwards the coupling between the problems is introduced by exchanging the waveforms, 
and solving the system iteratively. 
In each waveform relaxation iteration $l$, the resolution of mesoscale problems $($for instance with 
solutions $\bah_\mathrm{c}^{(l)}(t))$, is followed by the resolution of the macroscale 
problem $($for instance with the solution $\bah_\mathrm{M}^{(l)}(t))$ until convergence is reached 
$($for instance \mbox{$\|\bah_\mathrm{M}^{(l-1)}-\bah_\mathrm{M}^{(l)}\|_{L^{\infty}(0,\, T; \Ltwo{\Omega}) } < tol_M)$}.
In the rest of the section,
we will often omit the time window index $n$ to simplify notation 
(for instance $\bah_{\mathrm{M}}^{(j, k, l, n)}$ and $\bah_{\mathrm{c}}^{(i, j, k, l, n)}$ become
$\bah_{\mathrm{M}}^{(j, k, l)}$ and $\bah_{\mathrm{c}}^{(i, j, k, l)}$, respectively). 

For any given waveform relaxation iteration $l$, $N_{\mathrm{GP}}$ mesoscale problems 
are solved (in parallel) using the macroscale source terms from the previous waveform iteration
$l-1$. 
In the following section we discuss these two problems starting with the macroscale.

\subsection{The macroscale problem}
\label{WR-FE-HMM-Macro}
The waveform relaxation starts with the resolution of mesoscale problems. 
The solutions $\bah_{\mathrm{c}}^{(i, k, l, n)}$ are then 
used for computing the homogenized constitutive law needed by the nonlinear 
macroscale problem derived from the semi-discrete equations \eqref{eq:newton-macro}: 
\begin{equation}
  \bM_{\mathrm{M}}
	\frac{\bah_\mathrm{M}^{(k, l)}-\bah_\mathrm{M}^{(k-1, l)}}{\Delta t_k}
	+
	\bF_{\mathrm{M}}\left(\bah_\mathrm{M}^{(k, l)},\bah_{\mathrm{c}}^{(k, l)}\right)
        =0
  \label{eq:newton-macro-WR}
\end{equation}with known (mesoscale) corrections $\bah_{\mathrm{c}}^{(k, l)}$ at time points $t_k$. 
The macroscale and the mesoscale problems are decoupled and the homogenized constitutive law
$\bH_\mathrm{M}^{(k,l)}$ used in \eqref{eq:newton-macro-WR} is upscaled using the formula
\begin{equation*}
\bH_M^{(k,l)}(\bx, t, \bb_M(\bx, t)) = 
\frac{1}{|\Omega_{\mathrm{m}}|} \int_{\Omega_{\mathrm{m}}} \bH(\bx, \by, \bb_{\mathrm{c}}^{(k,l)}(\bx, \by, t) + \bb_M(\bx, t)) \mathrm{d} \by.
\label{eq:Homogenized-Law-WR-FE-HMM}
\end{equation*}
with $\bb_\mathrm{M}=\Curl[_\mathrm{x}]{\ba_{\mathrm{M}}}$ and where the mesoscale field $\bb_\mathrm{c}^{(k,l)}=\Curl[_\mathrm{y}]{\ba_{\mathrm{c}}}^{(k,l)}$ is obtained by solving  
mesoscale problems for a waveform relaxation iteration $l$ as explained in Section \ref{WR-FE-HMM-Meso}.
The decoupling between the macroscale and the mesoscale problems allows to 
compute the homogenized Jacobian directly by
\begin{equation}
\frac{\partial \bH_\mathrm{M}^{(k,l)}}{\partial \bah_\mathrm{M}} 
= \frac{1}{|\Omega_{\mathrm{m}}|} \int_{\Omega_{\mathrm{m}}} \left( \frac{\partial \bH}{\partial \bb_{\mathrm{M}}}(\bb_{\mathrm{c}}^{(k,l)} + \bb_{\mathrm{M}})  \frac{\partial \bb_{\mathrm{M}}}{\partial \bah_{\mathrm{M}}} \right) \mathrm{d} \by,
\end{equation}
since $\bH$ is known as a closed-form expression. This results from the independence of the mesoscale solutions $\bah_{\mathrm{c}}^{(k,l)}$ on the 
macroscale source $\bah_{\mathrm{M}}^{(k,l)}$. Indeed, each mesoscale problem corresponding to 
the Gau{\ss} point denoted by $i$ is computed using the 
macroscale fields from the previous waveform relaxation iteration
\begin{equation*}
\bah_{\mathrm{c}}^{( k, l)} = \boldsymbol{\mathcal{A}}_{\mathrm{c}}(\bah_{\mathrm{M}}^{(k,l-1)}).
\label{eq:bc-wr-fe-hmm}
\end{equation*}
Unlike the mesoscale problem in \eqref{eq:bc-fe-hmm} which was strongly coupled 
with the macroscale problem,
the function $\boldsymbol{\mathcal{A}}_{\mathrm{c}}$ is evaluated at each $\bah_{\mathrm{M}}^{(i,k,l-1)}$ and there is no need to evaluate 
the derivative $\partial {\boldsymbol{\mathcal{A}}_{\mathrm{c}}}/\partial \bah_{\mathrm{M}}$
using the finite difference method as done in \eqref{eq:Homogenized-Law-FE-HMM-1}.
Note however that one mesoscale field computation per Gau{\ss} point is needed 
for each waveform relaxation iteration.

Let us consider the case of a quasi-linear law as e.g. Brauer's model \cite{Brauer_1981aa}:
\begin{equation*}
\bH(\bb) = \mu(\bb) \, \bb 
\end{equation*}
with $\bb = \bb_{\mathrm{M}} + \bb_{\mathrm{c}}^{(k, l)}$, 
the integrand in \eqref{eq:bc-wr-fe-hmm} is calculated using 
\begin{equation*}
\displaystyle \frac{\partial \bH(\bb)}{\partial \bb_{\mathrm{M}}} 
= 
\mu(\bb) 
+ \frac{\partial \mu}{\partial |\bb|^2} 
\frac{\partial |\bb|^2}{\partial \bb_{\mathrm{M}}} \otimes \bb
=  
\mu(\bb) 
+ 2 \frac{\partial \mu}{\partial |\bb|^2} \bb \otimes \bb
\end{equation*}
where $\otimes$ denotes the square dyadic product.

\subsection{Mesoscale problems}
\label{WR-FE-HMM-Meso}
Starting from the mesoscale semi-discrete equations \eqref{eq:newton-meso}
of Problem \ref{eq:discret_multiscale_problem}, 
the following nonlinear mesoscale problems are derived for the $l^{th}$ waveform iteration:
\begin{equation}
    \bM_{\mathrm{m}} \frac{\bah_{\mathrm{c}}^{(i, k, l)} - 
    \bah_{\mathrm{c}}^{(i, k-1, l)}}{\Delta t_k}
    +
    \bF_{\mathrm{m}} \Big(\bah_{\mathrm{c}}^{(i, k, l)}, \bah_{\mathrm{M}}^{(k, l-1)}\Big)    
    = 0, \quad i = 1, \ldots, N_{\mathrm{GP}}.
    \label{eq:newton-meso-WR}
\end{equation}
with known (macroscale) waveforms $\bah_{\mathrm{M}}^{(k, l-1)}$ (given at time points $t_k$).
In this equation, the macroscale source per Gau{\ss} point $\bah_\mathrm{M}^{(i, l-1)}$ 
is considered to be known and taken from the previous waveform relaxation 
iteration $l-1$. The mesoscale problems defined in \eqref{eq:newton-meso-WR} 
can then be solved in parallel on the time window $\mathcal{I}_n$ and the 
mesoscale solutions stored for all Gau{\ss} points of the macroscale grid. 
These solutions are later used for computing the homogenized material law and 
the Jacobian as described in Section \ref{WR-FE-HMM-Macro}. The mesocale corrections 
$\bb_{\mathrm{c}}$ appearing in \eqref{eq:Homogenized-Law-WR-FE-HMM}
is independent from the macroscale field $\bb_{\mathrm{M}}$.

If an implicit Euler scheme is used for the macroscale and the mesoscale problems, 
the overall discretized system consists in solving the following problem:
\begin{prob}[Nonlinear, discrete, WR multiscale problem]
Find a series of solutions
\begin{equation*}
  [\bah_{\mathrm{M}}^{(k, l)},\bah_{\mathrm{c}}^{(1, k, l)}, \ldots, \bah_{\mathrm{c}}^{(N_\mathrm{GP}, k, l)}] \in \mathrm{I\!R}^{N_\mathrm{M}+N_\mathrm{GP}\cdot N_\mathrm{c}}
\end{equation*}
such that
    \begin{align*}
     \bR_{\mathrm{M}}\left(\bah_{\mathrm{M}}^{(k, l)},\bah_{\mathrm{c}}^{(i, k, l)}\right)
     =   \bM_{\mathrm{M}}\frac{\bah_{\mathrm{M}}^{(k, l)} - \bah_{\mathrm{M}}^{(k-1, l)}}{\Delta t_k}
    +
    \mathcal{F}_{\mathrm{M}}\Big(\bah_{\mathrm{M}}^{(k, l)}, \bah_{\mathrm{c}}^{(k, l)}\Big) &= 0 \\
\intertext{and for the mesoscale problems $i=1,\ldots,N_\mathrm{GP}$}
    \bR_{\mathrm{m}}\left(\bah_{\mathrm{M}}^{(k,l-1)},\bah_{\mathrm{c}}^{(i, k, l)}\right)=
    \bM_{\mathrm{m}}
    \frac{\bah_{\mathrm{c}}^{(i, k, l)} - \bah_{\mathrm{c}}^{(i, k-1, l)}}{\Delta t_k}
    + 
    \mathcal{F}_{\mathrm{m}}\Big(\bah_{\mathrm{c}}^{(i, k, l)}, \bah_{\mathrm{M}}^{(k, l-1)}\Big) &= 0.
  \end{align*}
\label{eq:discrete_wr_problem}
\end{prob}
\noindent Let $\bah_{\mathrm{M}}^{(j,k,l)}$ and $\bah_{\mathrm{c}}^{(i,j,k,l)}$ denote the $j^{\mathrm{th}}$
Newton--Raphson iterates. Then we define 
\begin{equation}
    \bF_{\mathrm{M}}^{(j, k, l)} := \bF_{\mathrm{M}} \Big(\bah_{\mathrm{M}}^{(j, k, l)}, 
    \bah_{\mathrm{c}}^{(j, k, l)} \Big) 
    \quad , \quad
    \bF_{\mathrm{m}}^{(i, j, k, l)} := \bF_{\mathrm{m}} \left(\bah_{\mathrm{c}}^{(i, j, k, l)}, \bah_{\mathrm{M}}^{(i, j, k, l)},
    \frac{\bah_{\mathrm{M}}^{(i, j, k, l)} - \bah_{\mathrm{M}}^{(i, j, k-1, l)}}{\Delta t_k} \right).
\end{equation}
\noindent The following total Jacobian must be computed to resolve Problem 6.1:
\newline
\begin{equation}
J_{\bR}^{(j, k, l)} = 
\frac{1}{\Delta t_k}
\begin{pmatrix} 
    \displaystyle{\bM_{\mathrm{M}}\vphantom{\frac{\partial \bF^{(j,k)}_{\mathrm{M}}}{\partial \bah_{\mathrm{M}}^{(j,k)}}}}& 
    \displaystyle{0} & 
    \displaystyle{\cdots} & 
    \displaystyle{0} 
    \\ 
    \displaystyle{0} & 
    \displaystyle{\bM_{\mathrm{m}}\vphantom{\frac{\partial \bF^{(j,k)}_{\mathrm{M}}}{\partial \bah_{\mathrm{M}}^{(j,k)}}}} & 
    \displaystyle{0} & 
    \displaystyle{0}    
    \\ 
    \displaystyle{\vdots} & 0 & \displaystyle{\ddots} & 0 
    \\ 
    \displaystyle{0} & 
    0 & 
    0 & 
    \displaystyle{\bM_{\mathrm{m}}\vphantom{\frac{\partial \bF^{(j,k)}_{\mathrm{M}}}{\partial \bah_{\mathrm{M}}^{(j,k)}}}}
\end{pmatrix}+
\begin{pmatrix} 
    \displaystyle{\frac{\partial \mathcal{F}_{\mathrm{M}}^{(j, k, l)}}{\partial \bah_{\mathrm{M}}^{(j, k, l)} }} & 
    \displaystyle{\frac{\partial \mathcal{F}_{\mathrm{M}}^{(j, k, l)}}{\partial \bah_{\mathrm{c}}^{(1, j, k, l)}}} & 
    \displaystyle{\cdots} & 
    \displaystyle{\frac{\partial \mathcal{F}_{\mathrm{M}}^{(j, k, l)}}{\partial \bah_{\mathrm{c}}^{(N_{\mathrm{GP}}, j, k, l)}}} 
    \\ 
    0 & 
    \displaystyle{\frac{\partial \mathcal{F}_{\mathrm{m}}^{(1, j, k, l)}}{\partial \bah_{\mathrm{c}}^{(1, j, k, l)}}} & 
    0 & 
    0    
    \\ 
    \displaystyle{\vdots} & 0 & \displaystyle{\ddots} & 0 
    \\ 
    0 & 
    0 & 
    0 & 
    \displaystyle{\frac{\partial \mathcal{F}_{\mathrm{m}}^{(N_{\mathrm{GP}}, j, k, l)}}{\partial \bah_{\mathrm{c}}^{(N_{\mathrm{GP}}, j, k, l)}}}
\end{pmatrix}.
\label{eq:Jacobian-full-WR}
\end{equation}
\noindent The decoupling of mesoscale from the macroscale 
solution makes all the elements of the first column equal to zero except for 
$\partial \mathcal{F}_{\mathrm{M}}^{(j, k, l)}/\partial \bah_{\mathrm{M}}^{(j, k, l)} $. Let us state the following result for the sake of completeness:

\begin{prop}
The use of the Jacobian \eqref{eq:Jacobian-full-WR} for solving the waveform relaxation problem is equivalent to the resolution of the following decoupled system for each waveform relaxation iteration $l$:  
\begin{equation*}
\bar{J}_{\mathcal{F}}^{(j, k, l)} = 
\frac{\bM_{\mathrm{M}}}{\Delta t_k}
+
\frac{\partial \mathcal{F}_{\mathrm{M}}^{(j, k, l)}}{\partial \bah_{\mathrm{M}}^{(j, k, l)} }
\label{eq:Jacobian-schur-complement-WR}
\end{equation*}
\end{prop}
\begin{proof}
The application of the Schur complement to \eqref{eq:Jacobian-full-WR} leads to 
the conclusion.
\end{proof}
\begin{algorithm}
\SetKwInput{Input}{Input}
\SetKwInput{Output}{Output}
\LinesNumbered
\DontPrintSemicolon
\Input{macroscale source $\bj_{\mathrm{s}}$ and mesh.}
\Output{fields (macro/meso), global quantities.}
\Begin
{   
    \emph{\# begin loops over time windows (index $n$)} \\
    \For{$(n \gets 1$ \KwTo $N_{\mathrm{TW}})$}{
        \emph{\# begin waveform relaxation (WR) loop (index $l$)} \\
        \For{$(l \gets 1$ \KwTo $N_{\mathrm{WR}} )$}{
            \emph{\# 1. Parallel resolution of meso-problems (index $i$)} \\
            \For{$(i \gets 1$ \KwTo $N_{\mathrm{GP}} )$}{
                downscale the sources $\bah_{\mathrm{M}}^{(l-1)}$,
                \vspace{1mm}
                \\
                solve mesoscale problems (1 per Gau{\ss} point) on $[t_k, t_{k+1}]$,
                \vspace{1mm}
                \\
                save the solution,
            }
            \emph{\# 2. Resolution of the macro-problem on $[t_k, t_{k+1}]$} \\
            $t \gets t_k$, initialize the macro-field $\ba_{\mathrm{M}}|_{t_k} = \ba_{{\mathrm{M}}k}$, \\
            \emph{\# begin the macroscale time loop}(index $k$) \\
            \For{$(k \gets 1$ \KwTo $N_{\mathrm{TS}})$}{   
                \emph{\# begin the macroscale NR loop (index $j$)} \\
                \For{$(j \gets 1$ \KwTo $N_{\mathrm{NR}}^{\mathrm{M}} )$}{
                    \emph{\# parallel updating of the homogenized law (index $i$)} \\
                    \For{$(i \gets 1$ \KwTo $N_{\mathrm{GP}})$}{
                        Read the mesoscale fields $\bb_{\mathrm{m}}^{(l)}$, \\
                        update $\partial \bH_{\mathrm{M}}^{(l)}/\partial \bb_{\mathrm{M}}^{(l)}$ using \eqref{eq:Homogenized-Law-WR-FE-HMM}, \\
                        upscale the law $\bH_{\mathrm{M}}^{(l)}, \partial \boldsymbol{\mathcal{H}}_{\mathrm{M}}^{(l)}/ \partial \bb_{\mathrm{M}}^{(l)}$,
                    }
                    assemble the matrix and solve,\newline
                } 
            }
        }
    }
}
\caption{Pseudocode for the waveform relaxation FE-HMM.}
\label{alg:WR-FE-HMM}
\end{algorithm}

The implementation of the waveform relaxation method is illustrated in Algorithm
\ref{alg:WR-FE-HMM}.  We propose a Gau{\ss}--Seidel type scheme between the macro-
and the mesoscale, where the mesoscale is solved in parallel.  It
starts with a loop over time windows and for each time window, a waveform
relaxation loop involving weakly coupled macroscale and mesoscale problems is
carried out.  The mesoscale are solved in parallel on the time interval
$\mathcal{I}_n$ using the macroscale sources from the previous WR iteration
$\bah_{\mathrm{M}}^{(l-1)}$ in the mesoscale problems. Mesoscale solutions are
then stored for later use in the evaluation of the homogenized magnetic field
and of the Jacobian. Then the macroscale problem is solved on $\mathcal{I}_n$
until convergence. The resolution involves a time discretization that leads to a
nonlinear problem that is solved using the Newton--Raphson method with the
Jacobian evaluated using the previously stored mesoscale fields.

Ultimately, the solution of the entire multiscale problem is obtained as follows 
(only superscripts of the waveform relaxation and the time window are involved,
i.e., the fields involved are $\ba_{\mathrm{c}}^{(l, n)}$ and $\ba_{\mathrm{M}}^{(l, n)}$): 
the macroscale field $\ba_\mathrm{M}^{(0, n)}$ is initialized with constant extrapolation.
The mesoscale problem can then be solved for the first iteration, i.e., one obtains 
$\ba_{\mathrm{c}}^{(1, n)}$.
This is then used for successive waveform relaxation iterations 
\begin{equation*}
\begin{aligned}
&\ba_\mathrm{M}^{(0, 1)}&
&\rightarrow&
&\ba_\mathrm{c}^{(1, 1)} &
&\rightarrow&
&\ba_\mathrm{M}^{(1, 1)} &
&\rightarrow&
&\ba_\mathrm{c}^{(2, 1)} &
&\rightarrow&
&\ldots&
&\rightarrow&
&\ba_\mathrm{M}^{(N_{\mathrm{WR}}, 1)}&
\\
&\ba_\mathrm{M}^{(0, 2)} &
&\rightarrow&
&\ba_\mathrm{c}^{(1, 2)} &
&\rightarrow&
&\ba_\mathrm{M}^{(1, 2)} &
&\rightarrow&
&\ba_\mathrm{c}^{(2, 2)} &
&\rightarrow&
&\ldots&
&\rightarrow&
&\ba_\mathrm{M}^{(N_{\mathrm{WR}}, 2)}, &
\\
&&
&&
&&
&&
&&
&\vdots&
&&
&&
&&
&&
&&
\\
&\ba_\mathrm{M}^{(0, N_{\mathrm{TW}})} &
&\rightarrow&
&\ba_\mathrm{c}^{(1, N_{\mathrm{TW}})}&
&\rightarrow&
&\ba_\mathrm{M}^{(1, N_{\mathrm{TW}})}&
&\rightarrow&
&\ba_\mathrm{c}^{(2, N_{\mathrm{TW}})} &
&\rightarrow&
&\ldots&
&\rightarrow&
&\ba_\mathrm{M}^{(N_{\mathrm{WR}}, N_{\mathrm{TW}})}.&
\end{aligned}
\end{equation*}

In addition to the flexible use of 
different FE bases and meshes at both scales, this approach also provides a 
natural setting for the use of different integrators and time step sizes. 
Communication costs can also be reduced in the case of parallel computations. 
As a drawback, the number of iterations for solving both the macroscale and the 
mesoscale problems may increase. 
A rigorous convergence analysis requires a structural
analysis of the coupled problem \eqref{eq:newton-macro-semidiscrete}--\eqref{eq:newton-meso-semidiscrete}.
In particular, the DAE index \cite{hairer_2010} must be known to guarantee convergence, c.f. \cite{Crow_1994}. 
This analysis is beyond the scope of this paper. 
However, the total costs for the monolithic and the waveform relaxation methods are estimated in the next Section.

\section{Estimation of the computational cost}
\label{sec:computational-cost}
In this section, we evaluate and compare total computational costs of the monolithic and 
the waveform relaxation algorithms. The total cost comprises the computational 
and the communication costs, which differ from one algorithm to the other.
We use $C^{\mathrm{m}}_{\mathrm{sol}}, C_{\mathrm{com}}, C^{\mathrm{M}}_{\mathrm{ass}}$ 
and $C^{\mathrm{M}}_{\mathrm{sol}}$ for the one time step costs for the mesoscale computations (including Newton iterations), 
the mesoscale--macroscale communications, the macroscale assembling and the macrocale resolution, respectively.
$C^{\mathrm{m}}_{\mathrm{sol}}$ is the most expensive and dominant operation. 
The total number of time steps is denoted by $N_\mathrm{TS}$ and the number of time windows 
$N_\mathrm{TW}$, such that the computational costs for a time window with 
$N_\mathrm{TS}/N_\mathrm{TW}$ time steps are roughly $\frac{N_\mathrm{TS}}{N_\mathrm{TW}}C^{\mathrm{m}}_{\mathrm{sol}}$. 

\subsection{Monolithic HMM}
\label{subsec:mono-cost}

The total cost of the monolithic algorithm is given by
\begin{align}
C_{\mathrm{Mono}} &= C^{\mathrm{m}}_{\mathrm{Mono}} + C^{\mathrm{M}}_{\mathrm{Mono}},
\label{eq:monolithic-cost-total}
\intertext{and has two contributions: the mesoscale contribution}
C^{\mathrm{m}}_{\mathrm{Mono}} 
&= 
N_{\mathrm{TS}} \, N_{\mathrm{NR}}^{\mathrm{M}} \, N_{\mathrm{GP}} (N^{\mathrm{m}}_\mathrm{dim} C^{\mathrm{m}}_{\mathrm{sol}} + C_{\mathrm{com}}), 
\label{eq:monolithic-cost-total-meso}
\intertext{and the macroscale contribution}
C^{\mathrm{M}}_{\mathrm{Mono}} 
&= 
N_{\mathrm{TS}} \, N_{\mathrm{NR}}^{\mathrm{M}} \, (C^{\mathrm{M}}_{\mathrm{ass}} + C^{\mathrm{M}}_{\mathrm{sol}}),
\label{eq:monolithic-cost-total-macro}
\end{align} 
where $N_{\mathrm{TS}}$ is the total number of time steps, $N_{\mathrm{NR}}^{\mathrm{M}}$ and $N_{\mathrm{NR}}^{\mathrm{m}}$ 
are the average numbers of Newton--Raphson iterations for macroscale and mesoscale problems to converge, $N^{\mathrm{m}}_\mathrm{dim}$ is the number of problems solved to approximate the Jacobian and the other $N_\star$ are defined in Table \ref{tab:loops}.

\medskip

The total cost in \eqref{eq:monolithic-cost-total} can be understood from the 
Algorithms \ref{alg:FE-HMM-Macro} and \ref{alg:FE-HMM-Meso}. The overall problem
is discretized in $N_{\mathrm{TS}}$ macroscale time steps and a macroscale nonlinear system 
is solved for each time step using the Newton--Raphson scheme. Therefore, 
$N_{\mathrm{NR}}^{\mathrm{M}}$ nonlinear (Newton--Raphson) iterations are performed 
for each time step. Each of these iterations involves the evaluation of 
the Jacobian at $N_{\mathrm{GP}}$ Gau{\ss} points by solving $N^{\mathrm{m}}_\mathrm{dim}=3$ or $4$ mesoscale problems.
The macroscale linear system is then assembled and solved.
The mesoscale cost $C^{\mathrm{m}}_{\mathrm{sol}}$ involves the resolution of nonlinear mesoscale 
problems over one time step. $N_{\mathrm{NR}}^{\mathrm{m}}$ nonlinear iterations 
are needed for each time step at the mesoscale and 
$2 N_{\mathrm{TS}} N_{\mathrm{NR}}^{\mathrm{M}} \, N_{\mathrm{GP}}$
communications involving the parallel transfer of small chunks of information are needed for the overall time interval. 
The macroscale assembling and resolution costs can be disregarded
with respect to the mesosclae costs, therefore
\begin{equation*}
C_{\mathrm{Mono}} \approx C^{\mathrm{m}}_{\mathrm{Mono}}.
\label{eq:monolithic-cost-total-1}
\end{equation*} 

\subsection{Waveform relaxation HMM}
\label{subsec:wr-hmm-cost}
The total cost for the WR algorithm is given by:
\begin{equation}
C_{\mathrm{WR}} = C^{\mathrm{m}}_{\mathrm{WR}} + C^{\mathrm{M}}_{\mathrm{WR}}.
\label{eq:wr-cost-total}
\end{equation} 
In the general case where $N_{\mathrm{TW}}$ time windows 
$\mathcal{I}_n = (t_n, t_{n+1}]$ are used. The two contributions in \eqref{eq:wr-cost-total}
are the mesoscale cost 
\begin{align}
C^{\mathrm{m}}_{\mathrm{WR}} 
&= 
N_{\mathrm{TS}} \, N_{\mathrm{WR}} \, N_{\mathrm{GP}} \left(C^{\mathrm{m}}_{\mathrm{sol}} + \frac{N_{\mathrm{TW}}}{N_{\mathrm{TS}}}C_{\mathrm{com}}\right)
\\
&= N_{\mathrm{TS}} \, N_{\mathrm{WR}} \, N_{\mathrm{GP}} C^{\mathrm{m}}_{\mathrm{sol}} + N_{\mathrm{TW}} \, N_{\mathrm{WR}} \, N_{\mathrm{GP}}C_{\mathrm{com}}, 
\label{eq:wr-cost-total-meso}
\intertext{and the macroscale cost}
C^{\mathrm{M}}_{\mathrm{WR}} 
&= N_{\mathrm{TW}} \, N_{\mathrm{WR}} \, N_{\mathrm{NR}}^{\mathrm{M}} \left( N_{\mathrm{GP}} C^{\mathrm{m}}_{\mathrm{jac}} 
+ C^{\mathrm{M}}_{\mathrm{ass}} + C^{\mathrm{M}}_{\mathrm{sol}}\right). 
\label{eq:wr-cost-total-macro}
\end{align}  
Additionally, $C^{\mathrm{m}}_{\mathrm{jac}}$ is the 
cost for reading a pre-stored mesoscale field map and evaluating the Jacobian 
for all time steps of the time window using \eqref{eq:Homogenized-Law-WR-FE-HMM}
for each time step, and $N_{\mathrm{NR}}^{\mathrm{M}}$ is the average 
number of Newton--Raphson iterations for macroscale problems to converge. 
These costs can be made small compared to the mesoscale computational cost 
$C^{\mathrm{m}}_{\mathrm{dim}}$ and the communication cost $C_{\mathrm{com}}$
by the use of a smart implementation.

The total cost in \eqref{eq:wr-cost-total} can be understood from the 
Algorithm \ref{alg:WR-FE-HMM}. The overall problem
is discretized and solved on $N_{\mathrm{TW}}$ time windows and for each time window
$\mathcal{I}_n$, a waveform relaxation loop involving $N_{\mathrm{WR}}$ WR 
iterations during which mesoscale problems are solved and stored.
The communication cost involves the transfer of $ \frac{N_\mathrm{TS}}{N_\mathrm{TW}} \,  N_{\mathrm{GP}} \,  N_{\mathrm{NR}}$ 
communications of bigger chunks of informations for each time window.

The nonlinear macroscale problem is solved using the Newton--Raphson
scheme. Therefore, $N_{\mathrm{NR}}^{\mathrm{M}}$ nonlinear iterations are carried for each
time step and for each Newton--Raphson iteration, the Jacobian is computed for
$N_{\mathrm{GP}}$ Gau{\ss} points. This is done by reading mesoscale field maps
for each Gau{\ss} point and then evaluating the homogenized law using
\eqref{eq:Homogenized-Law-WR-FE-HMM}.  The macroscale linear system is then
assembled and solved.  The reading of mesoscale fields maps and the update of
the homogenized law are one of the leverage for accelerating computations in the
context of the waveform relaxation method.

Neglecting the macroscale assembling and resolution costs, equation \eqref{eq:wr-cost-total}
can be approximated by
\begin{equation*}
C_{\mathrm{WR}} \approx N_{\mathrm{TS}} \, N_{\mathrm{WR}} \, N_{\mathrm{GP}} 
\left(C^{\mathrm{m}}_{\mathrm{sol}} + 
\frac{N_{\mathrm{TW}} }{N_{\mathrm{TS}} } \left(N_{\mathrm{NR}}^{\mathrm{M}} C^{\mathrm{m}}_{\mathrm{jac}} + C_{\mathrm{com}}\right)
\right)
,
\label{eq:wr-cost-total-1}
\end{equation*} 
where $N_{\mathrm{TS}}$ is the total number of time steps. 
The following theorem allows to compare computational costs for the 
monolithic and the waveform relaxation approaches.
\begin{thm}
The computational costs for the monolithic and the waveform relaxation methods 
are respectively given by the following approximations:
\begin{align}
C_{\mathrm{Mono}} 
&\approx N_{\mathrm{TS}} N_{\mathrm{NR}}^{\mathrm{M}} N_{\mathrm{GP}} \left(
\, N^{\mathrm{m}}_{\mathrm{dim}} \, C^{\mathrm{m}}_{\mathrm{sol}} + C_{\mathrm{com}}\right),  
\label{eq:monolithic-cost-total-2}
\intertext{and}
C_{\mathrm{WR}} 
&\approx N_{\mathrm{TS}} \, N_{\mathrm{WR}} \, N_{\mathrm{GP}} 
\left(C^{\mathrm{m}}_{\mathrm{sol}} + 
\frac{N_{\mathrm{TW}} }{N_{\mathrm{TS}} } \left(N_{\mathrm{NR}}^{\mathrm{M}} C^{\mathrm{m}}_{\mathrm{jac}} + C_{\mathrm{com}}\right)
\right).
\label{eq:wr-cost-total-meso-2}
\end{align}  
\end{thm}
\begin{proof}
The theorem results from the developments of Sections \ref{subsec:mono-cost} and 
\ref{subsec:wr-hmm-cost}. The high level of parallelization as explained in 
Algorithms \ref{alg:FE-HMM-Macro}-\ref{alg:FE-HMM-Meso} and \ref{alg:WR-FE-HMM} 
results from the independence of mesoscale problems.
\end{proof}
\noindent Moreover, the two approaches can easily be parallelized.

\begin{rem}
The computational cost of the waveform relaxation method can be decreased
by minimizing the cost related to the reading of the mesoscale fields and the update 
of the homogenized law.
\end{rem}
\noindent Assume that there exists
$\kappa \in (0, 1)$ 
such that 
$$ \displaystyle{\frac{N_{\mathrm{TW}}}{N_{\mathrm{TS}}}}
N_{\mathrm{NR}}^{\mathrm{M}} C^{\mathrm{m}}_{\mathrm{jac}} = 
\kappa (C^{\mathrm{m}}_{\mathrm{sol}} + \displaystyle{\frac{N_{\mathrm{TW}}}{N_{\mathrm{TS}}}} C_{\mathrm{com}})$$
which is a reasonable assumption because the cost due to the computation of 
the material law using the pre-stored mesoscale maps on the whole time window  
is small compared to the cost due to the computation of mesoscale problems on the same time window
and the communication. Then the relationship
\begin{align*}
  \nonumber
N_{\mathrm{WR}} \left(C^{\mathrm{m}}_{\mathrm{sol}} + 
\frac{N_{\mathrm{TW}} }{N_{\mathrm{TS}} } \left(N_{\mathrm{NR}}^{\mathrm{M}} C^{\mathrm{m}}_{\mathrm{jac}} + C_{\mathrm{com}}\right)
\right) 
&= N_{\mathrm{WR}} \left( (1 + \kappa) C^{\mathrm{m}}_{\mathrm{sol}} + \frac{N_{\mathrm{TW}} }{N_{\mathrm{TS}} } \left(1 + \kappa\right) C_{\mathrm{com}} \right)
\\
&< 
N_{\mathrm{NR}}^{\mathrm{M}} \left(N^{\mathrm{m}}_{\mathrm{dim}} \, C^{\mathrm{m}}_{\mathrm{sol}} + C_{\mathrm{com}} \right)
\label{eq:comparison-mono-wr-2-indi}
\end{align*} 
between \eqref{eq:monolithic-cost-total-2} and \eqref{eq:wr-cost-total-meso-2} 
shows that the waveform relaxation method is more efficient
if 
\begin{equation}
  N_{\mathrm{WR}} < \frac{N^{\mathrm{m}}_{\mathrm{dim}}}{(1 + \kappa)} N_{\mathrm{NR}}^{\mathrm{M}}  \label{eq:comparison-mono-wr-2}
\end{equation}\noindent and each time window consists of at least $N_{\mathrm{TS}} > 2$ time steps which is a rather technical assumption.
As can be seen from relation \eqref{eq:comparison-mono-wr-2}, reducing the number
of time windows ($N_{\mathrm{TW}}$) reduces the communication cost between the mesoscale 
and macroscale problems. Additionally, the reduction of cost due to the evaluation of the Jacobian minimizes the overall cost of the waveform relaxation method.

\section{Application}
\label{sec:applications}
We use a soft magnetic composite (SMC) material to test the ideas developed in the previous sections.
An idealized 2D periodic SMC (with $20 \, \times \, 20$ grains) surrounded 
by an inductor is considered. 

For the first series of numerical tests we use the SMC structure depicted 
in Figure \ref{fig:smc_grains_a-v} (only $10 \, \times \, 10$ grains are shown).
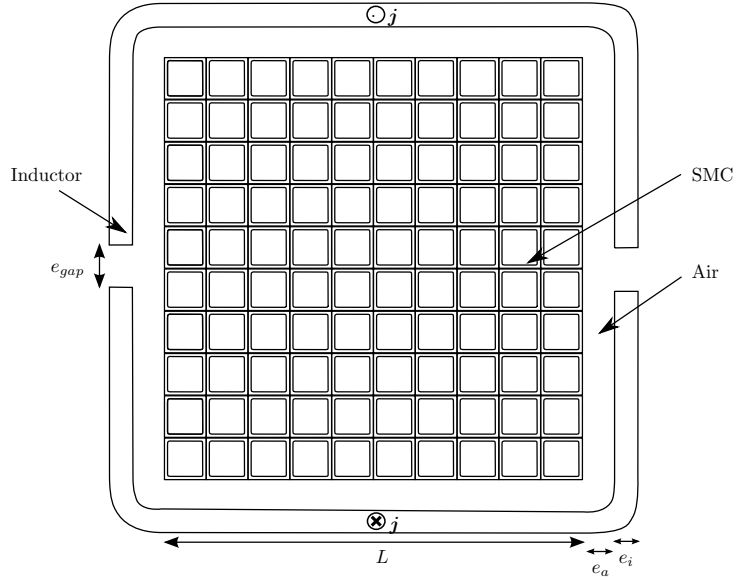
\begin{figure}
\centering
\scalebox{0.72}{\begingroup  \makeatletter  \providecommand\color[2][]{    \errmessage{(Inkscape) Color is used for the text in Inkscape, but the package 'color.sty' is not loaded}    \renewcommand\color[2][]{}  }  \providecommand\transparent[1]{    \errmessage{(Inkscape) Transparency is used (non-zero) for the text in Inkscape, but the package 'transparent.sty' is not loaded}    \renewcommand\transparent[1]{}  }  \providecommand\rotatebox[2]{#2}  \ifx\svgwidth\undefined    \setlength{\unitlength}{389.62924805bp}    \ifx\svgscale\undefined      \relax    \else      \setlength{\unitlength}{\unitlength * \real{\svgscale}}    \fi  \else    \setlength{\unitlength}{\svgwidth}  \fi  \global\let\svgwidth\undefined  \global\let\svgscale\undefined  \makeatother  \begin{picture}(1,0.77662753)    \put(0,0){\includegraphics[width=\unitlength]{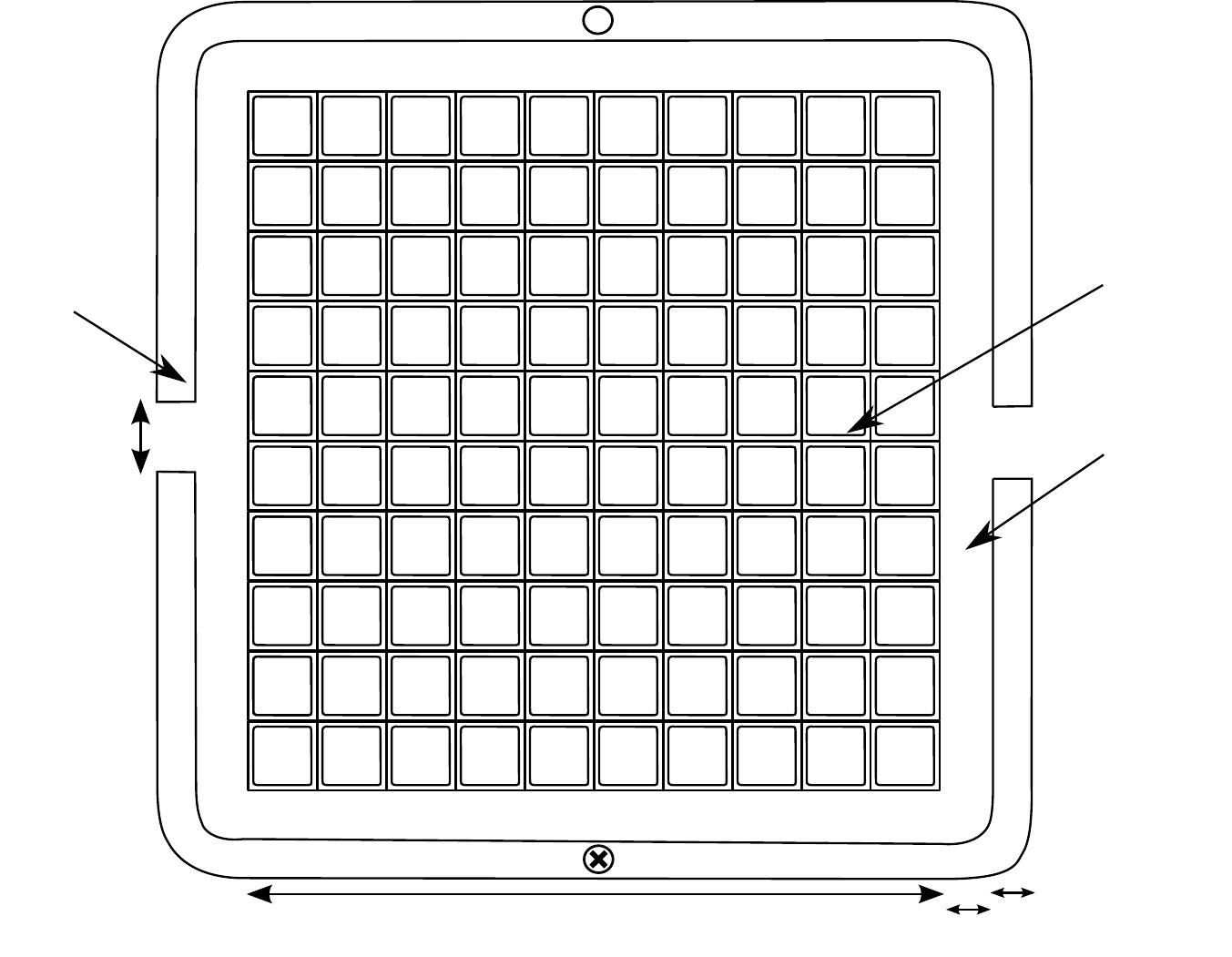}}    \put(-0.00493714,0.53689763){\color[rgb]{0,0,0}\makebox(0,0)[lb]{\smash{Inductor}}}    \put(0.91006439,0.5367161){\color[rgb]{0,0,0}\makebox(0,0)[lb]{\smash{SMC}}}    \put(0.91006439,0.40616586){\color[rgb]{0,0,0}\makebox(0,0)[lb]{\smash{Air}}}    \put(0.48545043,0.02269974){\color[rgb]{0,0,0}\makebox(0,0)[lb]{\smash{$L$}}}    \put(0.77571972,0.00906184){\color[rgb]{0,0,0}\makebox(0,0)[lb]{\smash{$e_a$}}}    \put(0.81165943,0.02269974){\color[rgb]{0,0,0}\makebox(0,0)[lb]{\smash{$e_i$}}}    \put(0.04891615,0.41437492){\color[rgb]{0,0,0}\makebox(0,0)[lb]{\smash{$e_{gap}$}}}    \put(0.47877082,0.75656306){\color[rgb]{0,0,0}\makebox(0,0)[lb]{\smash{.}}}    \put(0.50187631,0.7495362){\color[rgb]{0,0,0}\makebox(0,0)[lb]{\smash{$\bj$}}}    \put(0.50390869,0.06720575){\color[rgb]{0,0,0}\makebox(0,0)[lb]{\smash{$\bj$}}}  \end{picture}\endgroup }
\caption{Soft magnetic composite two-dimensional used geometry. 
Two opposite source current are imposed in the top and bottom inductors. 
The lengths are given by $L = 1000 \, \mu$m, $e_a = 150 \, \sqrt{2}/2 \, \mu$m, 
$e_i = 100 \, \mu$m and $e_{gap} = 100 \, \mu$m. Only 100 grains out of 400 are
drawn on the image.}
\label{fig:smc_grains_a-v}
\end{figure}
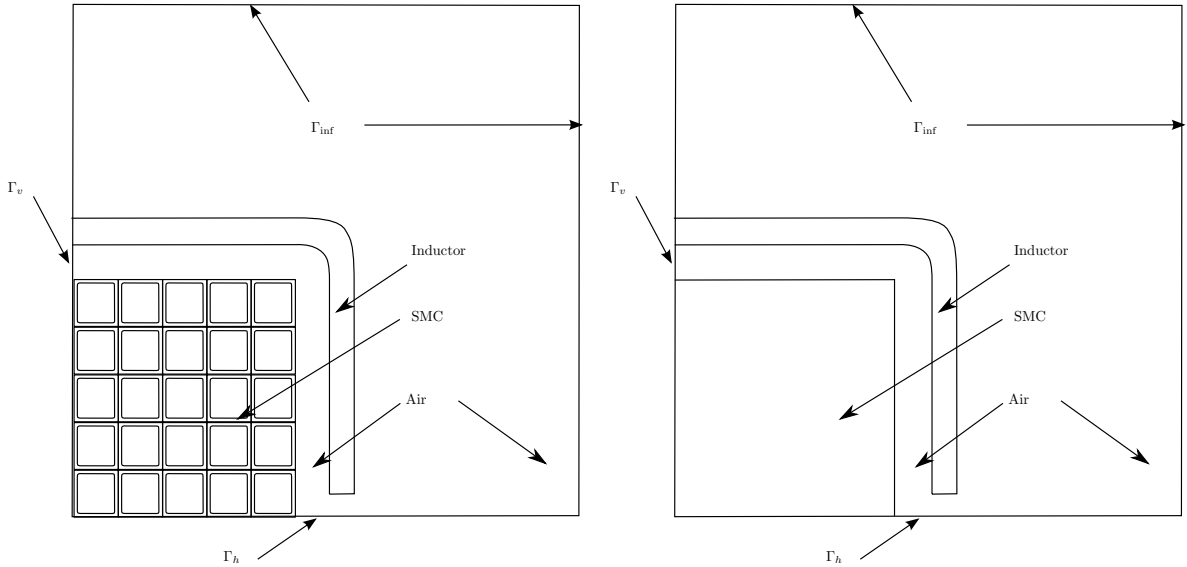
\begin{figure}
\centering
\scalebox{0.55}{\begingroup  \makeatletter  \providecommand\color[2][]{    \errmessage{(Inkscape) Color is used for the text in Inkscape, but the package 'color.sty' is not loaded}    \renewcommand\color[2][]{}  }  \providecommand\transparent[1]{    \errmessage{(Inkscape) Transparency is used (non-zero) for the text in Inkscape, but the package 'transparent.sty' is not loaded}    \renewcommand\transparent[1]{}  }  \providecommand\rotatebox[2]{#2}  \ifx\svgwidth\undefined    \setlength{\unitlength}{401.53049316bp}    \ifx\svgscale\undefined      \relax    \else      \setlength{\unitlength}{\unitlength * \real{\svgscale}}    \fi  \else    \setlength{\unitlength}{\svgwidth}  \fi  \global\let\svgwidth\undefined  \global\let\svgscale\undefined  \makeatother  \begin{picture}(1,0.96509316)    \put(0,0){\includegraphics[width=\unitlength]{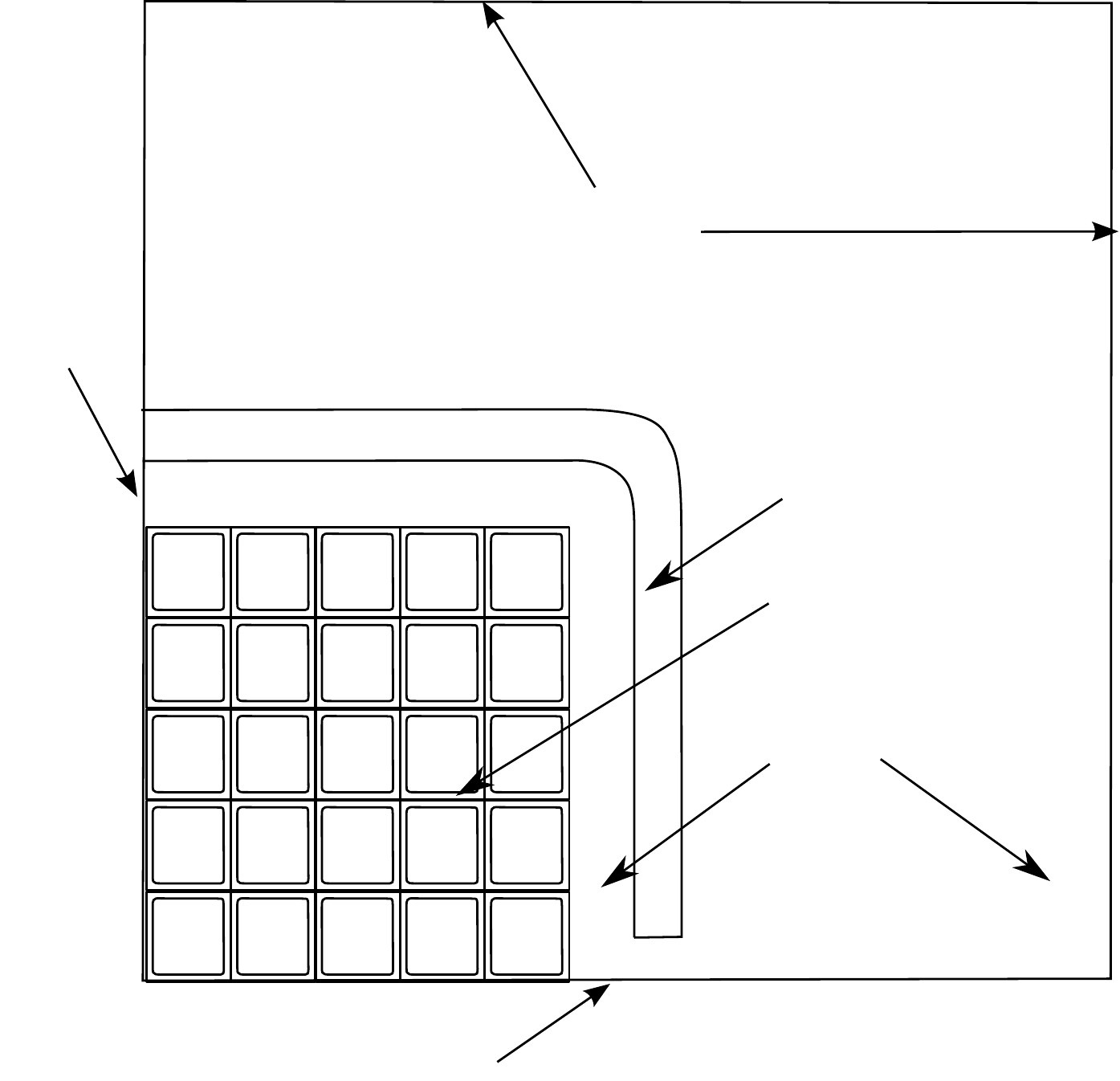}}    \put(0.70655616,0.53562561){\color[rgb]{0,0,0}\makebox(0,0)[lb]{\smash{Inductor}}}    \put(0.70655616,0.42376186){\color[rgb]{0,0,0}\makebox(0,0)[lb]{\smash{SMC}}}    \put(0.69704325,0.28086573){\color[rgb]{0,0,0}\makebox(0,0)[lb]{\smash{Air}}}    \put(0.01776985,0.64442937){\color[rgb]{0,0,0}\makebox(0,0)[lb]{\smash{$\Gamma_v$}}}    \put(0.53532352,0.74649089){\color[rgb]{0,0,0}\makebox(0,0)[lb]{\smash{$\Gamma_{\mathrm{inf}}$}}}    \put(0.38609155,0.01341198){\color[rgb]{0,0,0}\makebox(0,0)[lb]{\smash{$\Gamma_h$}}}  \end{picture}\endgroup }
\scalebox{0.55}{\begingroup  \makeatletter  \providecommand\color[2][]{    \errmessage{(Inkscape) Color is used for the text in Inkscape, but the package 'color.sty' is not loaded}    \renewcommand\color[2][]{}  }  \providecommand\transparent[1]{    \errmessage{(Inkscape) Transparency is used (non-zero) for the text in Inkscape, but the package 'transparent.sty' is not loaded}    \renewcommand\transparent[1]{}  }  \providecommand\rotatebox[2]{#2}  \ifx\svgwidth\undefined    \setlength{\unitlength}{393.53049316bp}    \ifx\svgscale\undefined      \relax    \else      \setlength{\unitlength}{\unitlength * \real{\svgscale}}    \fi  \else    \setlength{\unitlength}{\svgwidth}  \fi  \global\let\svgwidth\undefined  \global\let\svgscale\undefined  \makeatother  \begin{picture}(1,0.98471234)    \put(0,0){\includegraphics[width=\unitlength]{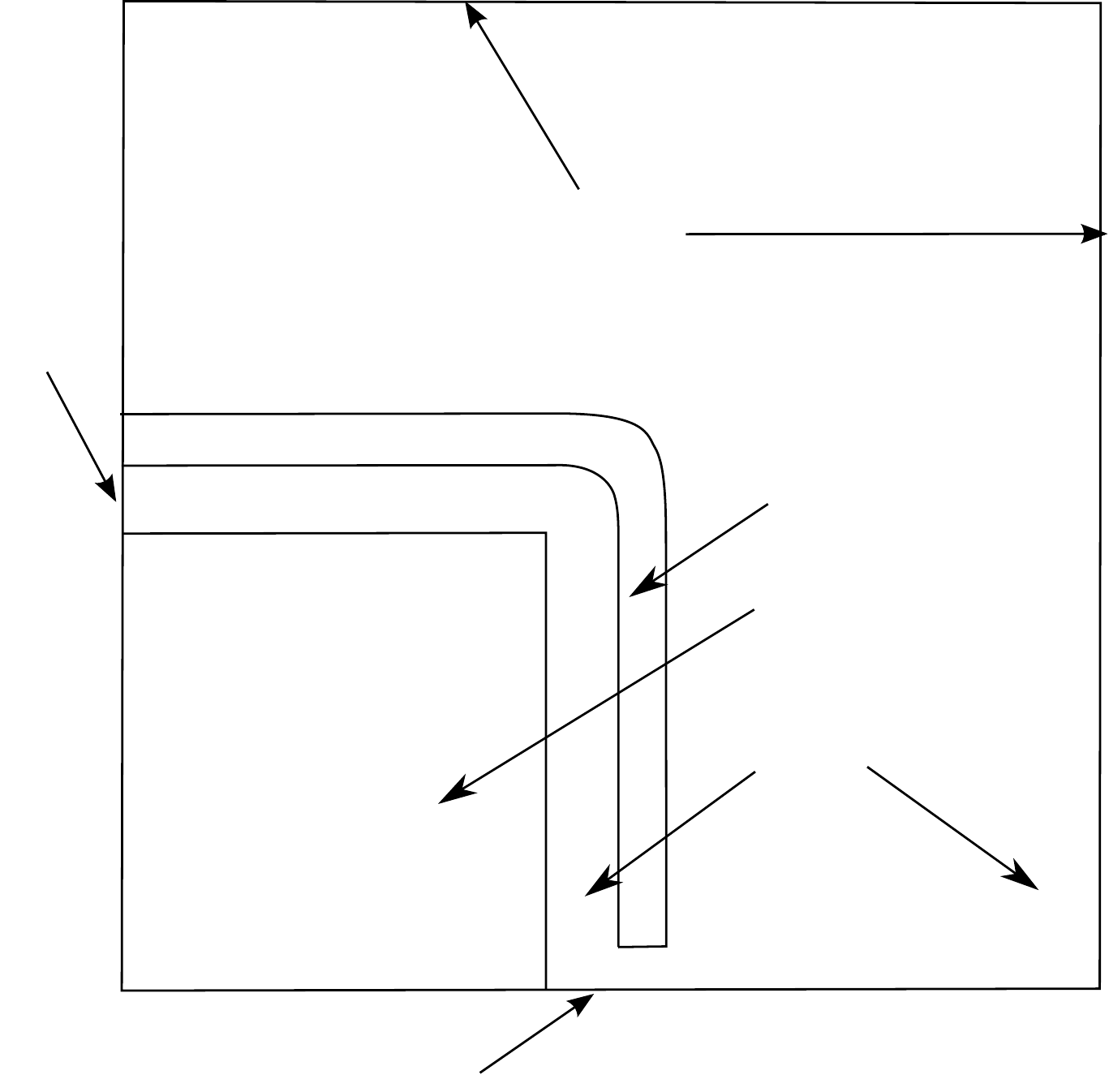}}    \put(0.70059081,0.54651424){\color[rgb]{0,0,0}\makebox(0,0)[lb]{\smash{Inductor}}}    \put(0.70059081,0.43237643){\color[rgb]{0,0,0}\makebox(0,0)[lb]{\smash{SMC}}}    \put(0.69088452,0.28657539){\color[rgb]{0,0,0}\makebox(0,0)[lb]{\smash{Air}}}    \put(-0.00219769,0.65752985){\color[rgb]{0,0,0}\makebox(0,0)[lb]{\smash{$\Gamma_v$}}}    \put(0.52587722,0.76166615){\color[rgb]{0,0,0}\makebox(0,0)[lb]{\smash{$\Gamma_{\mathrm{inf}}$}}}    \put(0.37361155,0.01368462){\color[rgb]{0,0,0}\makebox(0,0)[lb]{\smash{$\Gamma_h$}}}  \end{picture}\endgroup }
\caption{Geometry used for computations. Only a quarter of the geometry is used 
thanks to the symmetries. 
Left: reference geometry. Only 25 grains out of 100 are drawn on the image. 
Right: homogenized geometry.}
\label{fig:smc_grains_a-v_quarter}
\end{figure}
The geometry has been chosen such that the vector potential formulation 
$\bb = \Curl[]{\ba}$ as described in Section \ref{sec:mqs-problem} with 
$\ba = (0, 0, a_z)$ can be used.
The magnetic flux density $\bb = (b_x, b_y, 0)$ lives in the $xy$-plane. 
Only a quarter of the structure is considered for numerical computations thanks 
to the symmetry (see Figure \ref{fig:smc_grains_a-v_quarter} - left
for the reference geometry and Figure \ref{fig:smc_grains_a-v_quarter} - right for the 
geometry used for the homogenized problem).
In both cases, the following boundary conditions are imposed on 
$\Gamma_{\mathrm{inf}}, \Gamma_{h}$ and $\Gamma_{v}$:
\begin{equation*}
(\bn \cdot \bb) |_{\Gamma_{\mathrm{inf}}} = 0 \, \, \Leftarrow \, \, (\bn \times \ba) |_{\Gamma_{\mathrm{inf}}} = \boldsymbol{0}, \quad
(\bn \cdot \bb) |_{\Gamma_{\mathrm{h}}} = 0 \, \, \Leftarrow \, \,  (\bn \times \ba) |_{\Gamma_{\mathrm{h}}} = \boldsymbol{0}, \quad
(\bn \times \bh) |_{\Gamma_{\mathrm{v}}} = \boldsymbol{0}.
\label{eq:chap_results_BC}
\end{equation*}

Using Amp\`{e}re’s equation (\ref{eq:ampere-faraday-gauss}\,a),
the source current $\bj_{\mathrm{s}}$ must be imposed perpendicular to the $xy$-plane 
$\bj_{\mathrm{s}} = (0, 0, j_{\mathrm{s}})$ with $j_{\mathrm{s}} = j_{s0} s(t) = j_{s0} \sin(2 \pi f t)$. We consider the operating frequency $f = 50$ kHz which 
corresponds to $\lambda = 6000$m). The wavelength of the source is much larger
compared to the length of the structure ($\simeq 500 \mu m$) so that the 
assumption of a magnetoquasistatic problem can be made.

We consider isotropic materials and therefore, the magnetic field $\bh$ has only 
$xy$ components. We use the same material properties as those used in \cite{niyonzima-13}:
the insulation material is linear isotropic (with $\mu_r = 1$ and $\sigma
=0$). The conductor has an isotropic electric conductivity $\sigma = 5$\,MS and 
is also governed by the following nonlinear magnetic law \cite{Brauer_1981aa}: 
\begin{equation*}
\bH(\bb^{\varepsilon}) =  \Big( \alpha + \beta \, \exp(\gamma
||\bb^{\varepsilon}||^2) \Big)\,   \bb^{\varepsilon}
\end{equation*}
with $\alpha=  388$, $\beta = 0.3774$ and $\gamma = 2.97$.
The reference solution is obtained by solving a FE problem on an 
extremely fine mesh (347,324 elements) of the whole SMC structure. The mesoscale 
problems are solved on a square elementary cell meshed with (4,215 elements).
Results obtained using the newly developed waveform relaxation (WR) approach 
subscripted WR are compared to the reference results (subscripted ``Ref'') obtained 
solving the reference problem \eqref{eq:weakform_magdyn_multiscale} on a very 
fine mesh and results of the monolithic approach (subscripted``Mono'') obtained 
solving problem \eqref{eq:classical_macro} and \eqref{eq:classical_meso}.

Quantities of interest (global quantities and errors) are defined and used for numerical validation. 
The global quantities are the reference, the monolithic and the WR eddy currents losses:
\begin{equation*}
\tau \mathrm{P}_{\mathrm{Ref}}(t)    = q(\ba^{\varepsilon}), \quad
\tau \mathrm{P}_{\mathrm{mono}}(t)   = q(\ba_m), \quad
\tau \mathrm{P}_{\mathrm{WR}}^{l}(t) = q(\ba_m^{l}),
\label{eq:ecl_all}
\end{equation*}
where the functionals $q$ is defined as:
\[ q(\bu) =
  \begin{cases}
    \displaystyle \int_{\Omega_{\mathrm{c}}} (\sigma |\partial_t \bu^{\varepsilon}(\bx, t)|) \, \mathrm{d} x 
    & \quad \text{if } \bu \text{ is the reference solution}
    \\
    \displaystyle \int_{\Omega} \left( \frac{1}{|\Omega_{\mathrm{m}}|} \int_{\Omega_{\mathrm{mc}}} 
    (\sigma |\partial_t \bu_{\mathrm{m}}(\bx, \by, t)|^2) \, \mathrm{d} y \right) \, \mathrm{d} x  
    & \quad \text{if } \bu \text{ is the solution of the multiscale method}
    \\
  \end{cases}
\]
Equivalent quantities can be defined in terms of the magnetic energy and the magnetic power.
Two types of errors are defined:
the relative error on global quantities
\begin{equation}
\mathrm{Err}_{\tau \mathrm{P}}^{l} = \delta_{\mathrm{rel}}(\tau \mathrm{P}_{\mathrm{WR}}^{l}, \tau \mathrm{P}_{\mathrm{Mono}}), \quad 
\mathrm{Err}_{\mathrm{Wmag}}^{l}   = \delta_{\mathrm{rel}}(\mathrm{Wmag}_{\mathrm{WR}}^{l}  , \mathrm{Wmag}_{\mathrm{Mono}}  ), 
\label{eq:error-jl-mage}
\end{equation}
and the relative error on the fields $\bu_{\mathrm{M}}$
\begin{equation}
\mathrm{Err}_{\bu}^{(l)} = \bar{\delta}_{\mathrm{rel}}(\bu_{\mathrm{M}}) 
\label{eq:error-fields}
\end{equation}
where $\bu$ stand for the fields $\bb_{\mathrm{M}}, \partial_t \ba_{\mathrm{M}}$ and $\partial_t \bb_{\mathrm{M}}$. The functions 
$\delta_{\mathrm{rel}}$ and $\bar{\delta}_{\mathrm{rel}}$ are defined by:
\begin{equation*}
\delta_{\mathrm{rel}}(v, w) = \displaystyle{\frac{||v - w||_{L^{\infty}(0, T) } }{||w||_{L^{\infty}(0, T) } }}, 
\quad 
\bar{\delta}_{\mathrm{rel}}(\bu) = \displaystyle{\frac{||\bu^{l} - \bu^{l-1}||_{L^{\infty}(0, T; \bold{L}^{\infty}(\Omega) ) } }{||\bu^{0}||_{L^{\infty}(0, T; \bold{L}^{\infty}(\Omega) ) } } }.
\label{eq:ecl_all_functionals}
\end{equation*}

\subsection{Numerical convergence analysis}
The monolithic HMM and the WR HMM algorithms have been implemented in the open source 
software GetDP \cite{dular1998general} using a finite element formulation.
Two cases are  considered for the numerical validation of the method:
(a) the case with 1 time window and the same time stepping at the macroscale and the mesoscale;
(b) the case with 1 time window and different time stepping at the macroscale and the mesoscale.

\begin{figure}
    \begin{tikzpicture}
	\begin{axis}[width=0.5\textwidth,height=6cm,xmin=0,xmax=4e-5,ymin=0,ymax=6,xlabel={Time (s)},ylabel={Eddy current losses (W)},y label style={at={(axis description cs:0.07,0.5)}}]
	   \addplot+[color=black,only marks,mark=*,mark size=0.7,mark options=solid] table [x=a, y=b, col sep=space] {results/result0.dat};
	   \addplot+[color=magenta,thick,mark=none,mark options={solid},smooth] table [x=a, y=b, col sep=space]      {results/result1.dat};
	   \addplot+[color=blue, dashed, thick,mark=none,mark options={solid},smooth] table [x=a, y=b, col sep=space]      {results/result2.dat};
	   \addplot+[color=teal,thick,mark=none,mark options={solid},smooth] table [x=a, y=b, col sep=space]        {results/result3.dat};
	   	   \legend{Ref, Mono, WR $l=1$, WR $l=2$}
	\end{axis}
    \end{tikzpicture}
    \begin{tikzpicture}
	\begin{axis}[width=0.5\textwidth,height=6cm,xmin=0,xmax=4e-5,ymin=0,ymax=1e-1,xlabel={Time (s)},ylabel style={align=center,at={(axis description cs:0,0.5)}}, ylabel=Magnetic energy (J)\\~\\~,]
    \addplot+[color=black,only marks,mark=*,mark size=0.7,mark options=solid] table [x=a, y=b, col sep=space] {results/result4.dat};
	   \addplot+[color=magenta,thick,mark=none,mark options={solid},smooth] table [x=a, y=b, col sep=space]      {results/result5.dat};
	   \addplot+[color=blue, dashed, thick,mark=none,mark options={solid},smooth] table [x=a, y=b, col sep=space]      {results/result6.dat};
	   \addplot+[color=teal,thick,mark=none,mark options={solid},smooth] table [x=a, y=b, col sep=space]        {results/result7.dat};
	   	   \legend{Ref, Mono, WR $l=1$, WR $l=2$}
	\end{axis}
    \end{tikzpicture}
    \caption{Instantaneous eddy current losses (left) and magnetic power (right) for the reference, 
    the monolithic and the WR approaches. An overall mesh of 8722 elements with 25 elements for 
    the homogenized domain and a time step $\Delta t = 2\cdot10^{-7}$s were used. Only results 
    for the first two WR iterations are shown.}
    \label{fig:wr_joule_losses}
\end{figure}
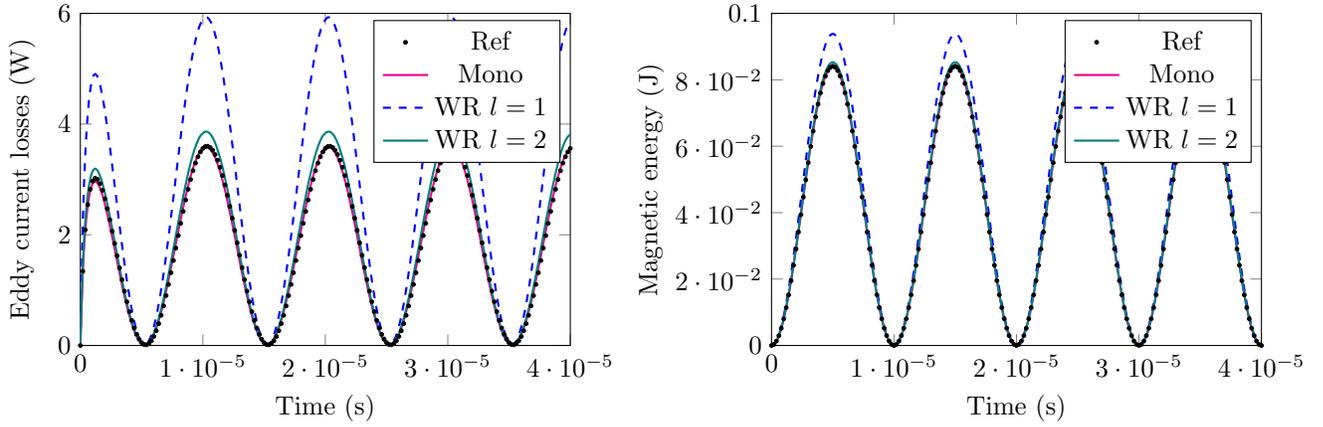
\begin{table}
\caption{Convergence of the eddy current losses and the magnetic energy as a function of 
the WR iterations. The relative errors $\mathrm{Err}_{\mathrm{\tau \mathrm{P}}}$ and 
$\mathrm{Err}_{\mathrm{\mathrm{Wmag}}}$
between the reference and the monolithic approaches are $2.12 \cdot 10^{-2}$ and $1.24 \cdot 10^{-2}$, respectively.}
\label{tab:convergence_edl}
\medskip\centering
\begin{tabular}{| c | c | c | c | c |}
\hline
WR iteration $l$   
& $\mathrm{Err}_{\mathrm{\tau \mathrm{P}}}^{l}$ 
& $\mathrm{Err}_{\mathrm{\tau \mathrm{P}}}^{l} - \mathrm{Err}_{\mathrm{\tau \mathrm{P}}}^{l-1}$
& $\mathrm{Err}_{\mathrm{\mathrm{Wmag}}}^{l}$ 
& $\mathrm{Err}_{\mathrm{\mathrm{Wmag}}}^{l} - \mathrm{Err}_{\mathrm{\mathrm{Wmag}}}^{l-1}$\\
\hline
 $1$       &  $6.68 \cdot 10^{-1}$  & $-$                    &  $1.17 \cdot 10^{-1}$ & $-$                  \\
 $2$       &  $8.08 \cdot 10^{-2}$  & $5.87 \cdot 10^{-1}$  &  $1.45 \cdot 10^{-2}$ & $1.03 \cdot 10^{-1}$\\
 $3$       &  $1.17 \cdot 10^{-2}$  & $6.90 \cdot 10^{-2}$  &  $2.08 \cdot 10^{-3}$ & $1.24 \cdot 10^{-2}$\\
 $4$       &  $1.87 \cdot 10^{-3}$  & $9.84 \cdot 10^{-3}$  &  $3.18 \cdot 10^{-4}$ & $1.77 \cdot 10^{-3}$\\
 $5$       &  $3.78 \cdot 10^{-4}$  & $1.49 \cdot 10^{-3}$  &  $5.01 \cdot 10^{-5}$ & $2.68 \cdot 10^{-4}$\\
 $6$       &  $1.41 \cdot 10^{-4}$  & $2.35 \cdot 10^{-4}$  &  $8.00 \cdot 10^{-6}$ & $4.21 \cdot 10^{-5}$\\
 $7$       &  $1.05 \cdot 10^{-4}$  & $3.80 \cdot 10^{-5}$  &  $1.30 \cdot 10^{-6}$ & $6.70 \cdot 10^{-6}$\\
 $8$       &  $9.0 \cdot 10^{-5}$   & $6.00 \cdot 10^{-6}$  &  $2.00 \cdot 10^{-6}$ & $1.10 \cdot 10^{-6}$\\
 $9$       &  $9.8 \cdot 10^{-5}$   & $ < 10^{-6}$           &  $ < 10^{-6}$          & $ < 10^{-6}$         \\
 $10$      &  $9.8 \cdot 10^{-5}$   & $ < 10^{-6}$           &  $ < 10^{-6}$          & $ < 10^{-7}$         \\
\hline
\end{tabular}
\end{table}
For case (a), an excellent agreement is obtained between the WR solutions and
the monolithic solutions to which the monolithic solutions are expected to
converge for the eddy current losses (Figure \ref{fig:wr_joule_losses} - left)
and the magnetic energy (Figure \ref{fig:wr_joule_losses} - right).  
Table \ref{tab:convergence_edl} depicts the evolution of the relative $L^{\infty}$
error defined in \eqref{eq:error-jl-mage} on eddy currents losses and the
magnetic energy between the WR and the monolithic approach.  As can be seen from
this table, the error between the WR and the monolithic cases is smaller than
the error between the monolithic and the reference cases after the third WR
iteration.  These findings were verified for all the numerical tests that were
run.
\begin{figure}[ht!]
    \begin{tikzpicture}
	\begin{semilogyaxis}[width=0.5\textwidth,height=6cm,xmin=1,xmax=10,ymax=1.0,xlabel={Waveform relaxation $l$},ylabel={Err$_{\mathrm{\bold{f}}}^l$} ]					
	   \addplot table [x=a, y=b, col sep=space] {results/result8.dat};
	   \addplot table [x=a, y=b, col sep=space] {results/result9.dat};
	   \addplot table [x=a, y=b, col sep=space] {results/result10.dat};
	   \legend{$\bb_M$, $\partial_t \ba_M$, $\partial_t \bb_M$}
	\end{semilogyaxis}
    \end{tikzpicture}
    \begin{tikzpicture}
	\begin{semilogyaxis}[width=0.5\textwidth,height=6cm,legend pos=north east,xmin=1,xmax=10,ymin=4.5e-7,ymax=1.1e-1,xlabel={Waveform relaxation $l$},ylabel={$\mathrm{Err}_{\mathrm{\mathrm{Wmag}}}^{l}$} ]					
	   \addplot table [x=a, y=b, col sep=space] {results/result11.dat};
	   \addplot table [x=a, y=b, col sep=space] {results/result12.dat};
	   \addplot table [x=a, y=b, col sep=space] {results/result13.dat};
	   \legend{$TS_{10}$, $TS_{25}$, $TS_{50}$}	
	 \end{semilogyaxis}
    \end{tikzpicture}    
    \caption{Left: Convergence of the macroscale waveforms for successive WR iterations 
    (case with 1 time window and the same time discretization at the macroscale and the mesoscale).
    Right: Relative error on magnetic energy as a function of the WR iterations. The macroscale time grid comprise 
    10, 25 and 50 time steps per period, respectively whereas the time grid for the mesoscale problems
    comprise 50 time steps per period. In both cases, a macroscale mesh with 8722 elements with 25 elements for 
    the homogenized domain was used. The relative error between the reference and the monolithic plots 
    is 0.01243.}
    \label{fig:wr_1tw_conv_field}
\end{figure}
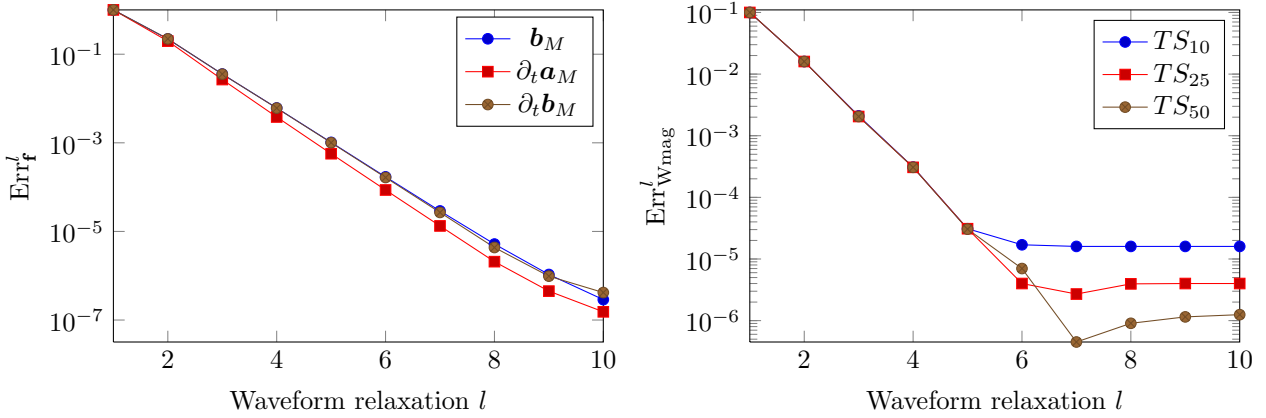
Figure \ref{fig:wr_1tw_conv_field} - left. Left} depicts the evolution of the relative $L^{\infty}$ 
error on the electromagnetic fields defined in \eqref{eq:error-fields} as a function of the WR iterations.
This criterion is to be used to control the error in case the monolithic solution
is not available.
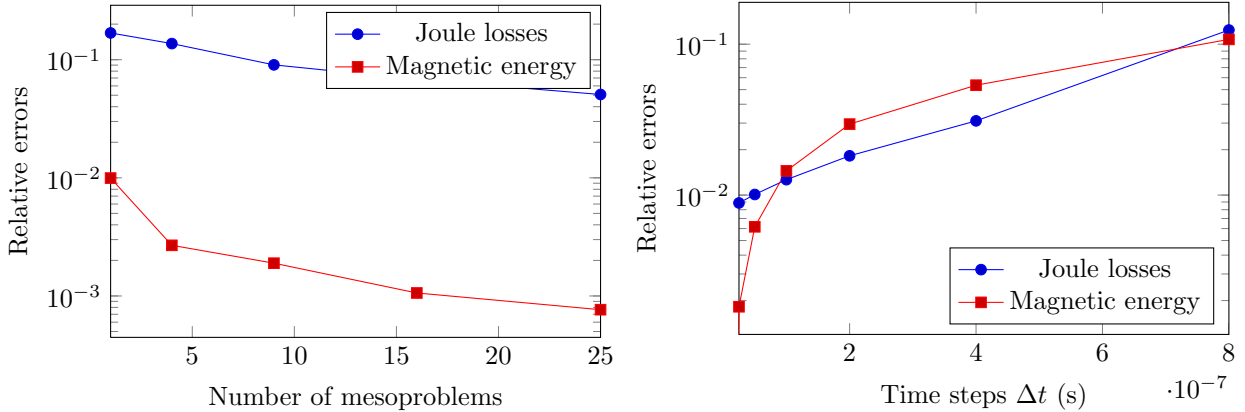
\begin{figure}
    \begin{tikzpicture}
	\begin{semilogyaxis}[width=0.5\textwidth,height=6cm,xmin=1,xmax=25,xlabel={Number of mesoproblems},ylabel={Relative errors} ]					
	   \addplot table [x=a, y=c, col sep=space] {results/result14.dat};
	   \addplot table [x=a, y=c, col sep=space] {results/result15.dat};
	   \legend{Joule losses, Magnetic energy}
	\end{semilogyaxis}
    \end{tikzpicture}
    \begin{tikzpicture}
	\begin{semilogyaxis}[width=0.5\textwidth,height=6cm,xmin=2.5e-8,xmax=8e-7,legend pos=south east,xlabel={Time steps $\Delta t$ (s)},ylabel={Relative errors} ]					
	   \addplot table [x=b, y=c, col sep=space] {results/result16.dat};
	   \addplot table [x=b, y=c, col sep=space] {results/result17.dat};
	   \legend{Joule losses, Magnetic energy}
	\end{semilogyaxis}
    \end{tikzpicture}    
    \caption{Relative errors on eddy current losses and magnetic energy between the waveform 
    relaxation and the reference cases. Left: Relative errors as a function of the number of mesoproblems
    (1 to 25 mesoproblems are considered for the homogenized domain).
    Right: Relative errors as a function of the time discretization. 
    The time step ranges from $2.5\cdot10^{-8}$s to $8\cdot10^{-7}$s.}
    \label{fig:wr_space_time_meshes_conv_field}
\end{figure}

Also for case (b) a good agreement is obtained between the waveform relaxation
solutions and the monolithic solutions.  
Figure \ref{fig:wr_1tw_conv_field} - right 
illustrates the case with different
macroscale time discretization (10, 25 and 50 time steps per period) but with
the same time discretization at the mesoscale (50 time steps per period). In
particular, a good agreement is observed for values of the error up to
$10^{-4}$. Beyond this value, smaller errors are even obtained for finer
macroscale time grids.  This underlines the possibility for efficient and
consistent usage of multirate time stepping

Figure \ref{fig:wr_space_time_meshes_conv_field} depicts the convergence of the 
eddy current losses and the magnetic field to the reference solutions when the 
spatial grid and time grid are refined. As can be seen in Figure 
\ref{fig:wr_space_time_meshes_conv_field} - left, the relative errors decrease 
as the number of mesoproblems is increased from 1 to 25. 
One Gauss point was considered for each element and therefore the 
macroscale mesh for the homogenized domain contains the same number of elements. 
Figure \ref{fig:wr_space_time_meshes_conv_field} - right depicts the same evolution 
for the case for different time discretizations. In this case, a linear convergence 
is observed for the eddy current losses (with the time derivative) whereas the 
curve for magnetic energy exhibits a faster convergence.
A good agreement was also observed in the case of many time windows.

An empirical comparison of the computational cost between the two methods can be
made on the basis of the formula \eqref{eq:comparison-mono-wr-2}.  In our
numerical computations, we have always found that the errors on eddy currents
losses and on the magnetic energy ($\mathrm{Err}_{\tau \mathrm{P}}^{l}$ and
$\mathrm{Err}_{\mathrm{Wmag}}^{l}$) become smaller than the errors between
monolithic and the reference quantities already at the third waveform relaxation
iteration (see Table \ref{tab:convergence_edl}).  The first iteration is not
computationally costly as it involves the initialization of the mesoscale
solution to zero. Therefore, it is reasonable to assume that the errors of both
methods become comparable for $N_{\mathrm{WR}} = 2$.  
Neglecting the mesoscale costs related to the reading and the updating of the homogenized constitutive
law (i.e.: $\kappa \rightarrow 0$) and considering
$N_{\mathrm{NR}}^{\mathrm{m}} = 3$ for a residue of $10^{-6}$ 
so that $N_{\mathrm{NR}}^{\mathrm{m}} = N_{\mathrm{NR}}^{\mathrm{M}}$, then $N_{\mathrm{WR}} < N_{\mathrm{NR}}^{\mathrm{M}} N^\mathrm{m}_\mathrm{dim}$ 
with $N_{\mathrm{WR}} = 2$, $N_{\mathrm{NR}}^{\mathrm{M}} = 3$, $N^\mathrm{m}_\mathrm{dim} = 4$ 
for a three-dimensional problem (resp. $N^\mathrm{m}_\mathrm{dim} = 3$ for a two-dimensional problem)
and a theoretical speed-up of 4.5 for two-dimensional problems (resp. 6 for three-dimensional problems)
can be gained. However, the current proof-of-concept implementation only allows a speed up of 2. 

\section{Conclusion}

In this paper the heterogeneous multiscale method was combined with the waveform
relaxation method. An efficient algorithm exploiting exact Jacobian information
based on Schur-complements was proposed. Estimates have shown that an optimal
implementation of the algorithm can be expected to be up to 6 times faster than a
comparably monolithic approach. In the case of multirate behavior, 
even higher speed-up are expected. Convergence and efficiency have been numerically
investigated using a challenging test example. Finally, optimization of our
implementation and applying the available convergence analysis of waveform
relaxation for higher index differential algebraic systems is subject of future
research.

\section*{Acknowledgment}
This work was supported by the German Funding Agency (DFG) by the grant 
`Parallel and Explicit Methods for the Eddy Current Problem' (SCHO-1562/1-1), 
the `Excellence Initiative' of the German Federal and State Governments and the
Graduate School CE at Technische Universit\"at Darmstadt and the Belgian Science
Policy under grant IAP P7/02 (Multiscale modelling of electrical energy system).

\end{document}